\documentclass[11pt,reqno]{amsart}

\numberwithin{equation}{section}
\usepackage{times}
\usepackage{amsmath,amsfonts,amstext,amssymb,amsbsy,amsopn,amsthm,eucal}
\usepackage{txfonts}
\usepackage{dsfont}
\usepackage{graphicx}   % for figures
\usepackage[bookmarks=true, pdfstartview=FitH, colorlinks=true,citecolor=blue, linkcolor=blue]{hyperref}
\usepackage{accents}
\usepackage{enumerate}

\usepackage{color}

\newcommand{\Ric}{{\rm Ric}}

\newcommand{\diam}{{\rm diam}}

\newcommand{\Alex}{\text{Alex\,}}

\newcommand{\Alexnk}{\text{Alex}^n(\kappa)}

\newcommand{\dN}{\mathds{N}}

\newcommand{\dR}{\mathds{R}}
\newcommand{\dS}{\mathds{S}}
\newcommand{\dZ}{\mathds{Z}}

\newcommand{\bx}{\text{\bf{x}}}

\newcommand{\cH}{\mathcal{H}}

\newcommand{\cS}{\mathcal{S}}

\newcommand{\cW}{\mathcal{W}}

\newcommand{\red}[1]{\textcolor{red}{#1}}

\newcommand{\Bad}{\textsl{Bad\,}}

\newcommand{\tang}[3]{\tilde\measuredangle\left({#1}\,_{#3}^{#2}\right)}

\newtheorem{theorem}{Theorem}[section]

\newtheorem{proposition}[theorem]{Proposition}
\newtheorem{lemma}[theorem]{Lemma}
\newtheorem{corollary}[theorem]{Corollary}
\theoremstyle{definition}
\newtheorem{definition}[theorem]{Definition}
\theoremstyle{remark}
\newtheorem{remark}{Remark}[section]
\theoremstyle{remark}
\newtheorem{example}{Example}[section]
\theoremstyle{remark}

\theoremstyle{remark}\newtheorem{conjecture}{Conjecture}[section]
\theoremstyle{remark}

\begin{document}

\title{Quantitative Estimates on the Singular Sets of Alexandrov Spaces}

\author{Nan Li}
\address{N. Li, Department of Mathematics, The City University of New York - NYC College of
Technology, 300 Jay St., Brooklyn, NY 11201}
\email{NLi@citytech.cuny.edu}

\author{Aaron Naber}
\address{A. Naber, Department Of Mathematics, Northwestern University, 2033 Sheridan Rd., Evanston, IL 60208-2370}
\email{anaber@math.northwestern.edu}

\thanks{The first author was supported by PSC-CUNY Grant 61533-0049. The second author was supported by NSF grant DMS-1809011.}

%\date{\today}
\maketitle

\begin{abstract}
Let $X\in\text{Alex}\,^n(-1)$ be an $n$-dimensional Alexandrov space with curvature $\ge -1$. Let the $r$-scale $(k,\epsilon)$-singular set $\mathcal S^k_{\epsilon,\,r}(X)$ be the collection of $x\in X$ so that $B_r(x)$ is not $\epsilon r$-close to a ball in any splitting space $\dR^{k+1}\times Z$.  We show that there exists $C(n,\epsilon)>0$ and $\beta(n,\epsilon)>0$, independent of the volume, so that for any disjoint collection $\big\{B_{r_i}(x_i):x_i\in \mathcal S_{\epsilon,\,\beta r_i}^k(X)\cap B_1, \,r_i\le 1\big\}$, the packing estimate $\sum r_i^k\le C$ holds. Consequently, we obtain the Hausdorff measure estimates $\mathcal H^k(\mathcal S^k_\epsilon(X)\cap B_1)\le C$ and $\mathcal H^n\big(B_r (\mathcal S^k_{\epsilon,\,r}(X))\cap B_1(p)\big)\leq C\,r^{n-k}$. This answers an open question in \cite{KLP17}.  We also show that the $k$-singular set $\mathcal S^k(X)=\underset{\epsilon>0}\cup\left(\underset{r>0}\cap\mathcal S^k_{\epsilon,\,r}\right)$ is $k$-rectifiable and construct examples to show that such a structure is sharp. For instance, in the $k=1$ case we can build for any closed set $T\subseteq \dS^1$ and $\epsilon>0$ a space $Y\in\text{Alex}^3(0)$ with $\mathcal S^{1}_\epsilon(Y)=\phi(T)$, where $\phi\colon\dS^1\to Y$ is a bi-Lipschitz embedding.  Taking $T$ to be a Cantor set it gives rise to an example where the singular set is a $1$-rectifiable, $1$-Cantor set with positive $1$-Hausdorff measure.
\end{abstract}

\tableofcontents

\section{Introduction}\label{s:intro}

%Let $\Alex^n(\kappa)$ be the collection of $n$-dimensional Alexandrov spaces with (sectional) curvature $\ge \kappa$. The aim of this paper is to study the quantitative stratifications of $X\in\Alexnk$. To compare Alexandrov spaces with limits of manifolds with lower Ricci curvature bounds, we would like to point out that the Gromov-Hausdorff limits of Riemannian manifolds with $\sec\ge\kappa$ are Alexandrov spaces with curvature $\ge \kappa$. However, not every Alexandrov space is isometric to a non-collapsed limit of Riemannian manifolds with uniform lower sectional curvature bound, due to some topological obstruction (see \cite{Kap07} and \cite{Pel91}). It is an open question whether every Alexandrov space is a collapsed limit of Riemannian manifolds with uniform lower sectional curvature bound.

Let $\Alex^n(\kappa)$ be the collection of $n$-dimensional Alexandrov spaces with (sectional) curvature $\ge \kappa$. The aim of this paper is to study the quantitative stratifications of $X\in\Alexnk$. Given $X\in\Alex^n(\kappa)$, it is known that the tangent cone $T_p(X)$ at every point $p\in X$ is a metric cone $C(\Sigma)$, where $\Sigma\in\Alex^{n-1}(1)$ and it is unique. The singular set $\cS(X)$ is the collection of points whose tangent cones are not isometric to $\dR^n$. It has a natural stratification $$\cS(X)=\cS^{n-1}(X)\supseteq\cS^{n-2}(X)\supseteq\dots\supseteq\cS^{1}(X)\supseteq\cS^{0}(X)\, ,$$ where
\begin{align}
	&\cS^k(X) \equiv \{p\in X: \text{ $T_p(X)$ is not isometric to $\dR^{k+1}\times C(\Sigma)$ for any metric space $\Sigma$}\}\, .
\end{align}

We may omit the $X$ and write for example $\cS^k=\cS^k(X)$ if it doesn't cause any ambiguity. Let us first state a notion of strong quantitative singular sets. We will then compare it with those used for the Ricci cases.

\begin{definition}[Quantitative splitting]\label{d:k-splitting} \quad
\begin{enumerate}
\item Given a metric space $Y$ and $k\in\dN$, we say that $Y$ is $k$-splitting if $Y$ is isometric to $\dR^k\times Z$ for some metric space $Z$.
\item Given a metric space $X$ we say that a metric ball $B_r(x)\subseteq X$ is $(k,\epsilon)$-splitting if there exists a $k$-splitting space $Y$ and $y\in Y$ such that $d_{GH}(B_r(x),B_r(y))\le\epsilon r$.
\end{enumerate}
\end{definition}

\begin{definition}[Strong quantitative singular sets]\label{d:qsingular}
Given $k,\epsilon,\,r>0$ and metric space $X$.
\begin{enumerate}
\item The $r$-scale $(k,\epsilon)$-singular set
\begin{align}
&\cS^k_{\epsilon,\,r}(X) \equiv \big\{x\in X: B_r(x)\text{ is not $(k+1,\epsilon)$-splitting} \big\}\,.
\label{d:qsingular.e1}
\end{align}
\item The $(k,\epsilon)$-singular set
\begin{align}
\cS^k_{\epsilon} \equiv\underset{r>0}\cap\cS^k_{\epsilon,\,r}= \big\{&x\in X: B_r(x)\text{ is not $(k+1,\epsilon)$-splitting for every $0<r\le 1$ } \big\}.
\end{align}
\end{enumerate}
\end{definition}

%It is a short exercise to check that by the definition we have
%\begin{align}
%\cS^k_{2\epsilon} \subseteq \underset{r>0}\cap\cS^k_{\epsilon,\,r}\subseteq \cS^k_{\epsilon}
%\text{ \qquad and \qquad}
%&\cS^k = \underset{\epsilon>0}\cup\cS^k_{\epsilon} =\underset{\epsilon>0}\cup\left(\underset{r>0}\cap\cS^k_{\epsilon,\,r}\right).
%\end{align}

It's easy to see that $\cS^k = \underset{\epsilon>0}\cup\cS^k_{\epsilon} =\underset{\epsilon>0}\cup\left(\underset{r>0}\cap\cS^k_{\epsilon,\,r}\right)$. A weaker notion of quantitative singular sets, which we will denote by $\cW\cS^{\,k}_{\epsilon,\,r}$, was introduced in \cite{CN13} for manifolds with lower Ricci curvature bounds, see (\ref{d:qs-ric.e1}) for a definition. A significance for (\ref{d:qsingular.e1}) is that it requires $B_s(x)$ to be $(k,\epsilon)$-non-splitting only at the scale $s=r$, but not for all $r\le s\le 1$ as required in (\ref{d:qs-ric.e1}). It is worth pointing out that notion (\ref{d:qsingular.e1}) is strictly stronger than (\ref{d:qs-ric.e1}) on manifolds with Ricci curvature bounds, while they are equivalent in some sense on Alexandrov spaces (see Section \ref{subsection:Strong and weak singularity}). The singular sets defined as in (\ref{d:qsingular.e1}) are not known to satisfy the estimates established in \cite{CJN18}, \cite{CN13} or \cite{CN15} for the Ricci cases.

%See more details in Section \ref{subsection:Strong and weak singularity}.
%In particular, we prove that there exist $c(n,\epsilon)$, $\delta(n,\epsilon)>0$ such that for any $X\in\Alex^n(-1)$, if $x\in\cS^k_{\epsilon,r}(X)$, then $B_s(x)$ is not $(k+1,\delta)$-splitting for all $s\ge 5r$; if $x\notin\cS^k_{\delta,r}(X)$, then $B_s(x)$ is $(k+1,\epsilon)$-splitting for all $s\le \frac15 r$ and $B_t(x)$ is $(k+1,\epsilon)$-symmetric for some $t\in[cr, r]$

It was proved in \cite{CC97-I} that if $X$ is a Gromov-Hausdorff limit of $n$-dimensional, $v$-noncollapsed Riemannian manifolds with $\Ric\ge-(n-1)$, then the Hausdorff dimension $\dim_{\cH}(\cW\cS^k)\le k$. Under the same assumptions, it was proved in \cite{CJN18} that for any $0<r,\, \epsilon\le 1$, there exists a constant $C(n,v,\epsilon)>0$ such that for any $p\in X$, it holds that
\begin{align}
  \text{vol}(B_r(\cW\cS^{\,k}_{\epsilon,\,r}(X))\cap B_{1/2}(p))\le C(n,v,\epsilon)r^{\,n-k}.
  \label{e:vol.est.ric}
\end{align}
It was also proved in \cite{CJN18} that $\cW\cS^{\,k}_{\epsilon}(X)$ is $k$-rectifiable.  For $X\in\Alexnk$, it is proved in \cite{BGP} that the Hausdorff dimension $\dim_\cH(\cS^k(X))\le k$, and it was asked in \cite{KLP17} wether the $(n-2)$-dimensional packing estimate holds for $\cS^{n-2}_\epsilon(X)$. In this paper, we prove the $k$-packing estimates and the $k$-rectifiability of $\cS^k_\epsilon(X)$ for every $0\le k\le n$.  Moreover, all of our estimate are independent on the volume of unit balls in $X$. Note that it is crucial to have a positive lower volume bound in \cite{CJN18}, \cite{CN13} and \cite{CN15}, to obtain estimates such as (\ref{e:vol.est.ric}) for manifolds with lower Ricci curvature bounds. It is not known whether the volume dependence can be removed for the Ricci cases.

%{\it
%\begin{enumerate}
%  \item Can the order $r^{\,n-k-\epsilon}$ in the estimate be improved to $r^{\,n-k}$? Note that the order $r^{n-k}$ is sharp if $\dim_H(\cS^k_\epsilon)=k$.
%  \item Is $\cS^k_\epsilon$ $k$-rectifiable?
%  \item For $k\ge 1$, away from a zero $\mathcal H^k$-measure subset, can $\cS^k_\epsilon$ be equipped with a $k$-manifold structure?
%\end{enumerate}
%}
%
%
%
%In this paper, we answer these questions for Alexandrov spaces.
%Our results are stronger than those results in [??] for non-collapsed limits of manifolds with lower Ricci curvature bounds, in the sense that our notion of quantitative singular sets is stronger, and our results is volume independent.
%Our first main result is the following packing estimate. %Note that unlike the Ricci case our estimates are independent of volume, thus they are also applicable to collapsed Alexandrov spaces.

\begin{theorem}[Packing estimate]\label{t:packing_estimate}  For any $n\in\dN$ and $\epsilon>0$ there exists $C=C(n,\epsilon)>0$ and $\beta=\beta(n,\epsilon)>0$ such that the following hold for any $(X,p)\in\Alex^n(-1)$. If $x_i\in \cS^k_{\epsilon, \, \beta r_i}(X)\cap B_1(p)$ and  $\{B_{r_i}(x_i)\}$ are disjoint with $r_i\le 1$ for all $i\in\mathds I$, then
\begin{align}\sum_{i\in\mathds I} r_i^k < C.\end{align}
In particular, if $x_i\in \cS^k_{\epsilon,r}(X)\cap B_1(p)$ and $\{B_{r}(x_i)\}$ are disjoint with $r\le 1$, then $|\mathds I|<Cr^{-k}$.

\end{theorem}

%Let's compare the above packing estimates with the similar results known for Ricci limit spaces.

%\begin{enumerate}
%  \renewcommand{\labelenumi}{(\roman{enumi})}
%    \setlength{\itemsep}{5pt}
%  \item If $x_i\in \cS^k_{\epsilon,r}(X)\cap B_1(p)$ and $\{B_{r}(x_i)\}$ are disjoint with $r\le 1$, then $|\{B_{r}(x_i)\}|<Cr^{-k}$.
%  \item If $x_i\in \cS^k_{\epsilon, \, r_i}(X)\cap B_1(p)$ and  $\{B_{Rr_i}(x_i)\}$ are disjoint with $r_i\le 1$, then $\sum r_i^k < C$.
%\end{enumerate}
%
%
%\begin{enumerate}
%  \renewcommand{\labelenumi}{(\roman{enumi})}
%    \setlength{\itemsep}{5pt}
%  \item $\Big|\Big\{i: r_i=r,\, x_i\in \cS^k_{\epsilon,\,r}(X)\cap B_1(p)\Big\}\Big|<Cr^{-k}$. In particular, $\Big|\cS^0_{\epsilon}(X)\cap B_1(p)\Big|<C$.
%  \item There exists $\eta(n,\epsilon)>0$ so that $\displaystyle\sum_{i\in\mathds I} r_i^k < C$, where $\mathds I=\Big\{i: r_i\le 1, \,x_i\in \cS^k_{\epsilon,\,\eta r_i}(X)\cap B_1(p)\Big\}$.
%\end{enumerate}

%\blue{In (ii), $x_i\in \cS^k_{\epsilon, \,r_i}(X)$ is revised to be $x_i\in \cS^k_{\epsilon, \,\epsilon' r_i}(X)$ (see the new remark and example below). Also added (i) since it will be needed in the inductive proof. }

%Such a theorem can be illustrated by the following example.

\begin{example}\label{exam:intro}
There exists Alexandrov spaces (in fact non-collapsed Gromov-Hausdorff limits of manifolds with $\sec\ge 0$) whose singular set is dense. Such a space was constructed in \cite{OS94}. Begin with a regular tetrahedron $X_1$ in $\dR^3$. Suppose convex polyhedra $X_k$ with triangular faces $\Delta_i$, $i=1,2\dots,4\cdot 3^{k-1}$ has been constructed. Let $x_i$ be the centroid of face $\Delta_i$. Let $y_i\in\dR^3$ so that $d(y_i,X_k)=d(y_i,x_i)=d_k^{\,i}>0$. Let $Y_i$ be the tetrahedron formed by $y^i$ and $\Delta_i$. Define $X_{k+1}=X_k\cup(\cup_i Y_i)$. The constants  $d^{\,i}_k=d^{\,i}_k(X_k)$ can be chosen small enough so that $X_{k+1}$ is convex. We have that $\partial X_k\in\Alex^2(0)$ for all $k$. Thus $\displaystyle Y=\lim_{i\to\infty}\partial X_k\in\Alex^2(0)$. It's easy to see that if all $X_k$ are convex, then $\underset{i}\max\{d_k^i\}\to 0$ as $k\to\infty$.

The set of singular points $\cS^0(Y)\supseteq\bigcup_{i,\,k}\{x^{\,i}_k\}$ is dense in $Y$. However, $|\cS^0_\epsilon|<N(\epsilon)$, asserted by Theorem \ref{t:packing_estimate}. For this example, we can get an explicit estimate using Gauss-Bonnet formula. For each $p\in Y$, we have that the tangent cone $T_p(Y)=C(\dS_\beta^1)$ with $0<\beta\le 1$. Let $\theta_p=2\pi \beta$ be the cone angle. Then we have $\cS^0_\epsilon=\Big\{p\in Y\colon \theta_p\le 2\pi-\epsilon\Big\}.$ Note that for any $p\in Y$ the Gaussian curvature $K_p\ge 0$  and $K_p=(2\pi-\theta_p)\delta_p$ if $p\in\cS^0_\epsilon$, where $\delta_p$ is the Dirac delta function at $p$. By Gauss-Bonnet formula, we have for $\epsilon_i=2^{-i}$
$$4\pi=\int_Y K\ge\sum_{i=0}^\infty\sum_{p\in\cS_{\epsilon_{i+1}}^0\setminus \cS_{\epsilon_i}^0}(2\pi-\theta_p)\geq \sum_{i=0}^\infty\epsilon_{i+1}\, |\cS_{\epsilon_{i+1}}^0\setminus \cS_{\epsilon_i}^0|.$$
In particular, we have the estimate $|\cS^0_\epsilon|\le \frac{4\pi}{\epsilon}$. $\square$
\end{example}

The statement of Theorem \ref{t:packing_estimate} (ii) is not true without a quantitative control of $\beta=\beta(n,\epsilon)$, if $\inf\{r_i\}=0$. See the following example.

\begin{example}
  Let $X=C(S^1_\rho)$ be a metric cone over a circle with radius $\rho=\frac1{20}$. Let $p$ be the cone point and choose points $x_i\in X$, so that $d(p,x_i)=3^{-i}$, $i=0, 1,2,\dots$. Consider disjoint collection $\mathcal C=\big\{B_{r_i}(x_i): r_i=\frac12\cdot 3^{-i}\big\}$. By the cone structure, we have $d_{GH}\big(B_{r_i}(x_i),\, Z\times[-r_i,r_i]\big)\ge\frac1{10}r_i\sin(\pi\rho)>\frac1{100}r_i$ for any metric space $Z$. Thus $x_i\in\cS^0_{\epsilon,\,r_i}(X)$ for any $0<\epsilon<\frac1{200}$. However, $|\mathcal C|=\infty$.
\end{example}

By a standard covering technique, Theorem \ref{t:packing_estimate} implies the following Hausdorff measure estimate.

\begin{corollary}[Hausdorff measure estimate]\label{t:Hk-measure bound}
For any $n\in\dN$ and $\epsilon>0$ there exists $C=C(n,\epsilon)>0$ such that for any $X\in\Alex^n(-1)$ and $p\in X$, we have the Hausdorff measure estimate
  \begin{align}
    \mathcal H^k\big(\cS^k_{\epsilon}(X) \cap B_1(p)\big)<C(n,\epsilon).
  \end{align}
\end{corollary}

We also have the following conjectural form of the constant in the above theorem:

\begin{conjecture}
  For any $(X,p)\in\Alex^n(-1)$, we have $$\cH^{k}(\cS^{k}_\epsilon\cap B_1(p))<C(n)\epsilon^{1-(n-k)}.$$
\end{conjecture}

Indeed, we may even have the following stronger summable form, see Example \ref{exam:intro}:

\begin{conjecture}
  For any $(X,p)\in\Alex^n(-1)$ and let $\epsilon_i =2^{-i}$, then we have \newline $$\sum_{i=0}^\infty \epsilon_{i+1}^{(n-k)-1}\cH^{k}\Big(\big(\cS^{k}_{\epsilon_{i+1}}\setminus\cS^{k}_{\epsilon_{i}} \big)\cap B_1(p)\Big)<C(n).$$
\end{conjecture}

Now let $\{B_{r}(x_i)\}_{i=1}^N$ be a Vitali covering of $B_r(\cS^k_{\epsilon,\,r}(X)\cap B_1(p))$ with $x_i\in \cS^k_{\epsilon,\,r}(X)$. By Theorem \ref{t:packing_estimate}, we have that $N\le C(n,\epsilon)\, r^{-k}$. Combining it with $\mathcal H^n(B_{r}(x))\le C(n)\,r^n$ for every $x\in X$ and $r\le 1$, we have the following estimate, which only matters in the noncollapsing setting:

\begin{corollary}[Volume estimate]\label{t:quant_strat}
For any $n\in\dN$ and $\epsilon>0$ there exists $C=C(n,\epsilon)>0$ such that the following estimate holds for any $X\in\Alex^n(-1)$ and $p\in X$.
\begin{align}
\mathcal H^n\big(B_r(\cS^k_{\epsilon,\,r}(X))\cap B_1(p)\big)\leq C\, r^{n-k}\, .
\end{align}
\end{corollary}
%\begin{remark}
%We would like to point out that for collapsed $X\in\Alex^n(-1)$, the uniform estimate
%$$\mathcal H^n\big(B_r(\cS^k_{\epsilon,\,r}(X))\cap B_1(p)\big)\leq C(n,\epsilon)\, \mathcal H^n(B_1(p))\,r^{n-k}$$ is not true.
%\end{remark}

%It is obvious that $B_r(\cS^k_{\epsilon,\,r}(X))$ collapses whenever $X$ collapses. However, the collapsing speed of $B_r(\cS^k_{\epsilon,\,r}(X))$  is not necessarily controlled by the collapsing speed of $X$. Namely, the following estimate is not necessarily true for every $(X,p,d)\in\Alex^n(-1)$ (see Example \ref{e:quant_strat_vol}).
%\begin{align}
%\mathcal H^n\big(B_r(\cS^k_{\epsilon,\,r}(X))\cap B_1(p)\big)\leq C(n,\epsilon)\,\mathcal H^n(B_1(p))\, r^{n-k}\, .
%\label{t:quant_strat_vol.e1}
%\end{align}

We also show that $\cS^k_\epsilon$ is $k$-rectifiable.

\begin{theorem}[$k$-rectifiability]\label{t:rectifiable}
   For any $X\in\Alex^n(-1)$ and $0\le k\le n$ we have that $\cS^k(X)$ is $k$-rectifiable.
\end{theorem}

It was asked for both Ricci and Alexandrov cases wether $\cS^k_\epsilon$ carries with a $k$-manifold structure, away from a zero $\mathcal H^k$-measure subset. It was proved in \cite{BGP} that for any $X\in\Alex^n(\kappa)$, if $p\in X\setminus\cS^{n-1}_\epsilon$, then there exists $r>0$ so that $B_r(p)$ is bi-Lipschitz to $B_r(0)\subset\dR^n$. If $p\in \cS^{n-1}_\epsilon\setminus\cS^{n-2}_\epsilon$, then there exists $r>0$ so that $B_r(p)$ is bi-Lipschitz to a ball centered at the origin in the half space $\dR^{n-1}\times \dR_{\ge 0}$. For $\cS^{n-2}_\epsilon$, we construct examples $X\in\Alex^n(\kappa)$ to show that it may contain no manifold point. %In this paper, we construct a counterexample $Y\in\Alex^3(0)$ for which $\cS^{1}_\epsilon(Y)$ contains no $1$-dimensional manifold point.

\begin{theorem}\label{t:sharp rect.cantor}
  For any closed subset $T\subseteq \dS^1$ and $\epsilon>0$, there exists a sequence of 3-dimensional manifolds $M_i$ with $\sec_{M_i}\ge 0$ and $M_i\to Y\in\Alex^3(0)$, for which $\cS^{1}_\epsilon(Y)=\phi(T)$, where $\phi\colon\dS^1\to Y$ is a bi-Lipschitz embedding.
\end{theorem}

In particular, $\cS^{1}_\epsilon(Y)$ can be a Cantor set with $\mathcal H^{1}(\cS^{1}_\epsilon(Y))>0$, which contains no manifold points. Let $Y_n=Y\times\dR^{n-3}\in \Alex^n(0)$. We have that $\cS^{n-2}_\epsilon(Y_n)$ contains no $(n-2)$-dimensional manifold point.  This shows that the rectifiable structure in Theorem \ref{t:rectifiable} is sharp.  Examples for which $\cS^k_\epsilon$ contains no $k$-manifold point, where $n\ge 4$ and $1\le k\le n-3$, can be similarly constructed, with a good amount of extra technical work.

\section{Outline of the Proof}

We begin with the notion of bad scales $\Bad^\epsilon(p)$. Fix a point $p\in X$ and $\epsilon>0$, then we define a $\dZ_2$-valued function $T^\epsilon_p(r,R)$ to describe the symmetry of metric balls $B_s(p)$ over scales $0\le r\le s\le R$. Define $T^\epsilon_p(r, R)=0$ if there exists a cone space $C(\Sigma)$, depending on $p, r, R, \epsilon$ but not on $s\in[r, R]$,  so that
  \begin{align}
    d_{GH}\Big(B_s(p), \, B_s(p^*)\Big)\le \epsilon s\, ,
    \label{defn:bad scale.e1}
  \end{align}
  for every $s\in[r, R]$, where $p^*\in C(\Sigma)$ is the cone point. Otherwise we define $T^\epsilon_p(r,R)=1$.  In the case that $T^\epsilon_p(r, R)=0$, we say that the metric ball $B_s(p)$ is uniformly $(0,\epsilon)$-symmetric for $r\le s\le R$.  It is clear that if $[a_1,a_2]\subseteq [r, R]$ and $T^\epsilon_p(r,R)=0$, then $T^\epsilon_p(a_1,a_2)=0$.  Contrapositively, if we have $[r, R]\subseteq[b_1,b_2]$ with $T^\epsilon_p(r,R)=1$, then we also have $T^\epsilon_p(b_1,b_2)=1$.

\begin{definition}[Bad scales]\label{d:bad scale}
   Let $r_\alpha=2^{-\alpha}$, where $\alpha\in\dN$. The {\it $\epsilon$-bad scales} $\{r_{\beta_{\,(j)}}\}\subseteq\{r_\alpha,\, \alpha\in\dN\}$ of $p$, denoted by $\Bad^\epsilon(p)$, are defined recursively as follows. Let $r_{\beta_{\,(0)}}=r_0=1$ and
  $$r_{\beta_{\,(k+1)}}=
  \renewcommand{\arraystretch}{1.5}
   \left\{\begin{array}{@{}l@{\quad}l@{}}
    r_{\beta_{\,(k)}+1}, & \text{ if \quad} T_p^\epsilon(r_{\beta_{\,(k)}+1}, r_{\beta_{\,(k)}})=1;
    \\
    r_\alpha, & \text{ if there exists } \alpha\ge\beta_{\,(k)}+1 \text{ such that } T_p^\epsilon(r_\alpha, r_{\beta_{\,(k)}})=0 \text{ but } T_p^\epsilon(r_{\alpha+1}, r_{\beta_{\,(k)}})=1.
    \end{array}\right.
  $$
\end{definition}

Note that if $[r,R]$ contains no $\epsilon$-bad scale of $p$, then $B_s(p)$ is uniformly $(0,\epsilon)$-symmetric for $r\le s\le R$.  This definition is strictly stronger than the corresponding definitions in the Ricci curvature context.

%It will be made rigorous in Definition \ref{d:bad scale}.  Roughly speaking, it is a collection of radii such that if $A_{s}^t(p)$ is an annulus around $p$ containing no $\epsilon$-bad scales of $p$, then there exists a fixed metric cone $C(Z)$ such that $B_r(p)$ is $\epsilon r$-close to $B_r(z^*)\subseteq C(Z)$ for all $s\leq r\leq t$, where $z^*\in Z$ is the cone point.

The following is a key lemma to build up our covering techniques.

\begin{lemma}[Finiteness of the number of bad scales]\label{l:finite_bad_radii}
For any $n\in\dN$ and $\epsilon>0$,  there exists $N(n,\epsilon)>0$ such that for any $(X,p)\in\Alex^n(-1)$, the number of $\epsilon$-bad scales satisfies $|\Bad^\epsilon(p)|<N(n,\epsilon)$.
\end{lemma}

The proof of this lemma is based on various point-wise monotonic properties of Alexandrov spaces. In particular, we prove Lemma \ref{t:GH-map on packing-Alex}, which we call ``almost packing cone implies almost metric cone".
It is an analogy of ``almost volume cone implies almost metric cone", which is the monotonic formula used for manifolds with lower Ricci curvature bound.  Note that both our monotonicity and the corresponding rigidity are strictly stronger than in the Ricci curvature context.

In order to state and prove our rigidity results we will need a splitting theory for Alexandrov spaces.

\begin{definition}[Strong splitting maps]\label{d:splitting map}
   Let $u_1, u_2,\dots, u_k\colon B_R(p)\to \dR$ be $\epsilon$-concave functions. The map $u=(u_1,\dots,u_k)\colon B_R(p)\to \dR^k$ is called a $(k, \epsilon)$-splitting map if the following are satisfied.
   \begin{enumerate}
  \renewcommand{\labelenumi}{(\roman{enumi})}
  \item $|\langle \nabla u_i,\nabla u_j \rangle-\delta_{ij}|\le\epsilon$.
  \item For any $x, y\in B_R(p)$ and any minimizing geodesic $\gamma$ connecting $x$ and $y$, it holds that $$\langle \uparrow_x^y, \nabla_x\, u_i \rangle + \langle\uparrow_y^x, \nabla_y\, u_i\rangle\le \epsilon.$$
      Here $\uparrow_{x}^{y}$ and $\uparrow_{y}^{x}$ denote the unit tangent directions of $\gamma$ at $x$ and $y$ respectively.
  \end{enumerate}
\end{definition}
\begin{remark}
  If $X$ is a smooth Riemannian manifold, the condition (ii) in the above definition says that on each geodesic, $u$ has a lower integral hessian bound.
\end{remark}

By the definition, we have that if $u\colon B_R(p)\to \dR^k$ is a $(k, \epsilon)$-splitting map, then $u|_{B_r}$ is also a $(k, \epsilon)$-splitting map for any $B_r\subset B_R(p)$. This restriction property of splitting maps is false in the context of manifolds with lower Ricci curvature bounds.  The existence and the properties of the strong splitting maps are discussed in Section \ref{subsec: Splitting theory}.\\

As in the standard dimension reduction, let us observe that for a metric cone $C(Z)$, the tangent cone of any point away from the cone tip splits off an extra $\dR$-factor comparing to $C(Z)$. We prove an effective version of this property in Lemma \ref{l:cone split}.

The monotonic property and the splitting theory lead to Theorem \ref{t:finite non-sym scale}. It says that there exist $\delta(n,\epsilon)$ and $\beta(n,\epsilon)>0$ so that if $u\colon B_{50}(p)\to \dR^k$ is a $(k,\delta)$-splitting function, and $\{B_{r_i}(x_i)\}$ is a disjoint collection with $x_i\in\cS^k_{\epsilon,\,\beta  r_i}$, then for any $z\in\dR^k$, we have
\begin{align}
  \Big|\Big\{i\in\mathds I\colon B_{\beta r_i}(x_i)\cap u^{-1}(z)\neq \varnothing\Big\}\Big|<N(n,\epsilon).
\end{align}

%\begin{align}
%  \Big|\Big\{i: x_i\in\cS^k_{\epsilon,\,\beta  r_i}(X)\cap u^{-1}(z)\cap B_1(p)\Big\}\Big|<N(n,\epsilon)
%\end{align}
%for any $z\in\dR^k$. This, together with the splitting theory developed in Section 5, implies Theorem \ref{t:finite non-sym scale}:
%\begin{align}
%  \Big|\Big\{x_i\in\cS^k_{\epsilon,\,\beta  r_i}\colon B_{\beta  r_i}(x_i)\cap u^{-1}(z)\neq \varnothing\Big\}\Big|<N(n,\epsilon).
%\end{align}

In particular, this Theorem implies that if we look at the associated collection of balls $\{B_{\beta r_i/4}( u(x_i))\}\subseteq \dR^k$ then its intersection number is at most $N(n,\epsilon)$.  That is, given any ball $B_{\beta r_j/4}(u(x_j))\in\{B_{\beta r_i/4}( u(x_i))\}$ it intersects at most $N-1$ other balls from the collection.  This shows that Theorem \ref{t:packing_estimate} holds if $B_1(p)$ is $(k,\epsilon)$-splitting.  We will then complete the proof by an induction on $k$.

%To finish the proof, we apply induction on $k$. Assume that $B_1(p)$ is not $(k+1,\epsilon)$-splitting, otherwise it has been proved by Lemma \ref{l:packing_estimate local}. Set $R_\alpha=2^{-\alpha}$, $\alpha\in \dZ$. Then for $\delta=\delta(n,\epsilon)>0$, the set $B_1(p)\setminus \cS^k_\delta\subseteq \cS^k_{\delta,10}\setminus \cS^k_\delta$ can be covered by balls $\{B_{\rho_j}(y_j), \rho_j\in\{\frac1{20}R_\alpha\}\}$, where $\{B_{\frac15\rho_j}(y_j)\}$ are disjoint, so that (1) every $B_{20\rho_j}(y_j)$ is not $(k+1,\delta)$-splitting, but (2) $B_{10\rho_j}(y_j)$ is $(k+1,\delta)$-splitting. Let $\mathds J_\alpha=\{j: \rho_j= \frac{1}{20}R_{\alpha}\}$ and $\mathds I_j=\{i: x_i\in B_{\rho_j}(y_j)\}$.
%The non-splitting condition (1) allows us to use the inductive hypothesis on $B_1(p)$ and get
%  \begin{align}
%     |\mathds J_\alpha|\le C(n,\epsilon) R_\alpha^{-k}.
%  \end{align}
%The splitting condition (2) allows us to use Lemma \ref{l:packing_estimate local} on each of $B_{\rho_j}(y_j)$ and get
%    \begin{align}
%     |\mathds I_j|\le C(n,\epsilon)(\rho_j/r)^{\,k+1}.
%  \end{align}
%Summing up these estimates we get Theorem \ref{t:packing_estimate} (i). The proof of (ii) is similar.

In Section \ref{s:examples} we construct examples to prove Theorem \ref{t:sharp rect.cantor}. Let us explain the moral of the construction below. The technical details will be added to make it rigorous in Section \ref{s:examples}.

 Let $Z=\bar B_1\subset \dR^2$ be a closed unit disk and $X_0=Z\times[0,1]\in\Alex^3(0)$ be a solid cylinder. For $\epsilon>0$ small, we have $\cS^0(X_0)= \varnothing$ and $\cS_\epsilon^1(X_0)=\partial Z\times\{0,1\}$ is a union of two unit circles.

%We define a concave function $f\colon Z\to \dR$, $z\mapsto \min\left\{\frac12,\, d(z,\partial Z)\right\}$. Let $X_0=\{(z,t)\in Z\times \dR: 0\leq t\leq f(z)\}$ be the subgraph of $f$. Then $X_0\in\Alex^3(0)$ and it can be viewed as a frustum cone over $Z$. Let $Z_t=\{z\in Z:f(z)\ge t\}$ be the sub-level sets in $Z$. For $\epsilon>0$ small, we have $\cS^0(X_0)= \varnothing$ and $\cS_\epsilon^1(X_0)=\left(\partial Z_{1/2}\times\{\frac12\}\right)\cup(\partial Z\times\{0\})$.

Now let $T\subseteq \partial Z$ be a closed subset, and thus $\partial Z\setminus T=\cup_\ell U_\ell$ is a collection of disjoint open intervals. Let $p$ be the center of $Z$ and define $\displaystyle C_\ell=\cup_{x\in U_\ell}\gamma_{px}$, where $\gamma_{x,y}$ denotes a line connecting $x$ and $y$, be the collection of sectors associated to the open sets $U_\ell$.  Let us observe for any $x\in \partial Z$ that the curvature at $\left(x, 1\right)\in X_0$ is $+\infty$ along the normal direction of $\partial Z\times\{1\}$ and strictly positive along its tangential direction.
%Note that $C_\ell\supset \bar U_\ell$ is convex.
This will allow us to smoothly ``sand off" each of $U_\ell\times\{1\}$ inside its convex hull $C_\ell\times[0,1]$, so that both the convexity of $X_0$ and the tangent cones at points in $X_0\setminus (\cup C_\ell\times[0,1])$  are preserved.  Let $X_1\in\Alex^3(0)$ be the resulted space.  In particular, the tangent cones at the points of $T\times\{1\}$ are preserved, and thus we have that $\cS^1_\epsilon(X_1)=(T\times\{1\})\cup(\partial Z\times\{0\})$.  Similarly, we can smooth near $\partial Z\times \{0\}$ in order to construct $X_2$ with $\cS^1_\epsilon(X_2)=T\times\{1\}$.  Now let $Y_2$ be the doubling of $X_2$, which is now a boundary free Alexandrov space $Y_2\in\Alex^3(0)$ for which $\cS(Y_2)=\cS^1_\epsilon(Y_2)=T$ and $\cS^0_\epsilon(Y_2)=\varnothing$. %Then $Y=Y_1\times\dR^{n-3}$ gives us the desired example for $k=n-2$. The construction will be more technical for $1<k<n-2$.

\vspace{.5cm}

\section{Monotonicity and Packing Numbers}

In this section we describe a monotone formula which plays an important role in the constructions of subsequent sections.

\begin{definition}[Packing]
Let $X$ be a metric space and $S\subseteq X$ with $\diam(S)<\infty$. For $\epsilon>0$, we say that a subset $\bx\equiv \{x_i\}\subseteq S$ is an {\it $\epsilon$-subpacking} if
\begin{equation}
 \text{$d(x_i,x_j)\geq \epsilon\, \diam(S) \text{ for every } i\neq j$.}
 \label{defn.enet1}
\end{equation}
An $\epsilon$-subpacking $\bx$ is said to be a packing if it is also $\epsilon\,\diam(S)$-dense in $S$.

%\red{We say $\bx\equiv \{x_j\}_1^N\in S$ is an $\epsilon$-covering if
%\begin{equation}
% \text{$\bx$ is $\epsilon\, \diam(S)$-dense in $S$}
% \label{defn.enet2}
%\end{equation}
%is satisfied. (We say $\bx\equiv \{x_j\}_1^N\in S$ is an $\epsilon$-net if both (\ref{defn.enet1}) and (\ref{defn.enet2}) are satisfied.)
%}

We write $|\bx|=N$ as the number of elements in $\bx$ if it is finite.  If we want to signify the set in question we may write $\bx=\bx(S)$.  We define the $\epsilon$-packing number $P_\epsilon(S)$ by
\begin{align}
P_\epsilon(S)\equiv \sup\{|\bx|:\,\bx\text{ is an $\epsilon$-subpacking for }S\}\, .
\end{align}
A packing $\bx$ is called a maximal $\epsilon$-packing of $S$ if $|\bx|=P_\epsilon(S)<\infty$.
\end{definition}

In the case that $S=B_r(p)$ is a metric ball we may write $\bx(p,r)=\bx(B_r(p))$ and
the $\epsilon$-packing number $P_\epsilon(p,r)\equiv P_\epsilon(B_r(p))$.  Let us record some easy but useful properties which hold for general metric spaces.

\begin{lemma}
Let $X$ be a metric space with $\epsilon>0$ fixed.  Then the following hold:
\begin{enumerate}
\renewcommand{\labelenumi}{(\roman{enumi})}
\item ({\it Enlargement}) If $\bx$ is an $\epsilon$-subpacking of $B_r(x)$, then either $\bx$ is an $\epsilon$-packing or there exists $x'\in B_r(x)$ such that $\bx'\equiv \bx\cup \{x'\}$ is also an $\epsilon$-subpacking.
\item ({\it Maximal subpacking $\implies$ packing}) If $\bx$ is an $\epsilon$-subpacking of $B_r(x)$ with $|\bx|=P_\epsilon(x,r)<\infty$, then $\bx$ is an $\epsilon$-packing.
\item ({\it $\epsilon$-monotonicity}) If $\bx$ is an $\epsilon$-subpacking and $\epsilon'<\epsilon$, then $\bx$ is an $\epsilon'$-subpacking. In particular, for each $r>0$ we have that $P_{\epsilon'}(x,r)\geq P_\epsilon(x,r)$.
\end{enumerate}
\end{lemma}

We wish to now discuss some more refined properties of $\epsilon$-packings and packing numbers for Alexandrov spaces.  To do this let us introduce the induced subpacking.  Indeed, this notation makes sense for any locally compact length metric space, but it is not so useful in general.

\begin{definition}[Induced subpacking]\label{d:induced subp}
Let $p\in X$, $R>0$ and for each $x\in\bar B_R(p)\setminus\{p\}$ we fix a geodesic $\gamma_{px}=\gamma^R_{px}$ connecting $p$ and $x$.  Given $0<r<R$, we define the inducting function $\varphi^{\,R}_r\colon \bar B_R(p)\to \bar B_r(p)$, $x\mapsto \bar x$, where $\bar x\in\gamma^R_{px}$ is the point with $d(p,\bar x)=\frac rR\cdot d(p,x)$. Now let $\{x_i\}_{i=1}^N$ be an $\epsilon$-subpacking of $\bar B_R(p)$ and $0<r<R$, then we call the collection of points $\{\varphi^{\,R}_r(x_i)\}_{i=1}^N$ the induced subpacking in $\bar B_r(p)$ of $\{x_i\}_{i=1}^N$.
\end{definition}

Note that the choice of geodesic $\gamma_{px}$ in the definition of $
\varphi^R_r$ is certainly not unique. However in the above definition of $\varphi^R_r$, such a choice is fixed for a given $R>0$ while independent of $0<r\leq R$.  If no confusion arises one may write $\gamma^R_{px} = \gamma_{px}$.

The proof of the following propositions are easy exercises based on the Toponogov comparisons.
\begin{proposition}\label{p:packings_basic_properties_alex}
Let $0<\epsilon, R< 1$. The following hold for any $(X,p)\in\Alex^n(-\epsilon)$ with $1\ge c\ge 1-R^2$. If $(X,p)\in\Alex^n(0)$, then $c\equiv 1$ can be chosen as a constant.
\begin{enumerate}
\renewcommand{\labelenumi}{(\roman{enumi})}
  \item ({\it Induced Packing}) Let $\{x_i\}$ be an $\epsilon$-subpacking of $B_R(p)$. For any $0<r<R$, the induced subpacking in $B_r(p)$ is a $c\epsilon$-subpacking.
  \item ({\it Monotonicity}) If $r\le R$, then the packing number $P_{c\epsilon}(x,r)\ge P_\epsilon(x,R)$.
  \item ({\it Bounds})  If $\bx=\{x_i\}$ is an $\epsilon$-subpacking for $B_r(p)$ with $0<r\leq 1$, then $1\leq |\bx|\leq C(n)\epsilon^{-n}$.  In particular, we have $1\leq P_\epsilon(p,r)\leq C(n)\epsilon^{-n}$.
  \item ({\it Density}) There exists a limit $\displaystyle\lim_{r\to 0} P_\epsilon(p,r)\equiv P_\epsilon(p)\le C(n)\epsilon^{-n}$, which we call the $\epsilon$-density at $x$. In fact, $P_\epsilon(p)=P_\epsilon(p^*,1)$, where $p^*$ is the cone point in the tangent cone at $p$.

%\item ({\it Equivalence})  Let $\bx$ be an $\epsilon$-packing of $B_r(p)$ with $0<r\leq 1$, then $C(n)^{-1}P_\epsilon(p,r)\leq|\bx|\leq P_\epsilon(p,r)$.

%\item (\red{\it Packing Scales-delete}) Given $x\in X$ there exists at most a finite number of radii $1\geq r^\epsilon_1(x)>\cdots>r^\epsilon_{N_x}(x)>r^\epsilon_{N_x+1}(x)=0$ with $N_x\leq C(n)\epsilon^{-n}$ such that $P_\epsilon(x,r^\epsilon_i(x))=P_\epsilon(x,r)<P_\epsilon(x,r^\epsilon_{j+1}(x))$ for $j=1,2,\dots,N_x$ and any $r\in(r^\epsilon_i(x), r^\epsilon_{j+1}(x))$.  We sometimes call the radii $r^\epsilon_i(x)$ the packing scales of $x$.
%\item (\red{\it Lower Semicontinuity-delete}) Given $x\in X$ we have that $\liminf_{y\to x} r_i^\epsilon(y)\geq r^\epsilon_i(x)$.
%\item (\red{\it Weak Upper Semicontinuity-delete})  Given $x\in X$ and any $\delta>0$ we have that $\limsup_{y\to x} r_i^{\epsilon-\delta}(y)\geq r^\epsilon_i(x)$.
\end{enumerate}
\end{proposition}
\vspace{.5cm}

\section{Bad Scales}

This section is dedicated to proving Lemma \ref{l:finite_bad_radii}. It says that there are at most a finite number of bad scales at each point, and our space has a fixed cone structure which persists over all good scales. Let us begin with an easy proposition.

\begin{proposition}\label{p:interval no bad}
  For any $0\le r< R/2<R\le 1$, if $(r,R)\cap \Bad^\epsilon(x)=\varnothing$, then $T_x^\epsilon(r,R)=0$.
\end{proposition}

\begin{proof}
The proof is almost taulogical.  Let $r_{\beta_{\,(k)}}=\inf\{r_\beta\in\Bad^\epsilon(x): r_\beta\ge R\}$ and $r_{\beta_{\,(k+1)}}$ be the next $\epsilon$-bad scale. Because $(r,R)\cap \Bad^\epsilon(x)=\varnothing$ we have that $r_{\beta_{\,(k+1)}}\le r< R/2<R\le r_{\beta_{\,(k)}}$. Therefore, $r_{\beta_{\,(k)}}/r_{\beta_{\,(k+1)}}>2$ and $\beta_{\,(k+1)}-\beta_{\,(k)}\ge 2$. By the definition  $T_x^\epsilon(r_{\beta_{\,(k+1)}}, r_{\beta_{\,(k)}})=0$. Note that $[r,R]\subseteq[r_{\beta_{\,(k+1)}}, r_{\beta_{\,(k)}}]$ and so we have that $T_x^\epsilon(r,R)=0$.
\end{proof}

%\begin{definition}
%  Let $(X,x,d)\in\Alex^n(-1)$. We define the $\delta$-bad scales $\{r_\alpha\}$ inductively. Let $r_0=1$. Suppose that $r_\alpha=(\delta/100)^{k_\alpha}$ is defined and $\Big\{\beta>k_\alpha:\, T^\delta(r_{k_\alpha}, r_\beta)=1\Big\}\neq\varnothing$. Let
%  \begin{align}
%    k_{\alpha+1}=
%    \min_{\beta> k_\alpha}\Big\{\beta :\, T^\delta(\beta, k_\alpha)=1\Big\}-1.
%  \end{align}
%  We define the bad scale $r_{j+1}=2^{-k_{j+1}}$. This recursive definition will stop in finitely many times due to the next lemma. We denote the set of $\delta$-bad scales $\{r_j\}$ of $x$ by $\Bad^\delta(x)$.
%\end{definition}

%\begin{remark}
%  By the definition, we have that if $0\le s<t\le 1$ and $(s, t)\cap \Bad^\delta(x)=\varnothing$, then $T^\delta_x(s,t)=0$, that is, there exists a metric space $Z$ and $z^*\in C(Z)$ being the cone point such that (\ref{defn:bad scale.e1}) holds for every $r\in[s, t]$.
%\end{remark}

%The next Theorem is a \red{key property} that will be used in the covering techniques:
%
%\begin{lemma}\label{l:finite_bad_radiii}
%For any $n\in\dN$ and $\delta>0$,  there exits $N(n,\delta)>0$ such that for any $(X, p, d)\in\Alex^n(-1)$, we have $|\Bad^\delta(p)|<N(n,\delta)$.
%\end{lemma}

To prove Lemma \ref{l:finite_bad_radii}, we need a result of the form ``almost packing cone implies almost metric cone". We begin with the following proposition. It follows directly from the definitions of $\epsilon$-packing and Hausdorff distance.

\begin{proposition}\label{l:GH-map on packing}
  Let $X$ and $Y$ be metric spaces whose diameters are both no more than 1. Let $\{x_i\}_{i=1}^{N_1}$ be an $\epsilon$-packing of $X$ and $\{y_i\}_{i=1}^{N_2}$ be an $\epsilon$-packing of $Y$. If $N_1=N_2=N$ and
  $$|d(x_i,x_j)-d(y_i,y_j)|\le\epsilon$$
  for every $1\le i,j\le N$, then $d_{GH}(X,Y)\le 4\epsilon$.
\end{proposition}

The first main result of this section is the following:\\

\begin{lemma}[Almost packing cone implies almost metric cone]\label{t:GH-map on packing-Alex}
  There is a universal constant $c>0$ such that the following holds for any $n\in\dN$ and $\epsilon\in(0,c)$. Let $(X,p)\in\Alex^n(-\epsilon)$ and $0\le r\le \frac12 R\le c$. Let $\bx(p,R)=\{x_i\}_{i=1}^N$ be a $\epsilon$-packing of $B_R(p)$. We have
    \begin{align}
    T_p^{\,\epsilon^{\,0.1}}(r,R)=0
    \label{l:GH-map on packing-Alex.e1}
  \end{align}
  if both of the following are satisfied.
   \begin{enumerate}
    \renewcommand{\labelenumi}{(\roman{enumi})}
    \item $P_{\epsilon}(p,r)=N=P_{\epsilon}(p,R)$.
    \item $r^{-1} d(\varphi^{\,R}_r(x_i),\,\varphi^{\,R}_r(x_j))\le R^{-1} d(x_i,x_j)+\epsilon$, for every $1\le i,j\le N$.
\end{enumerate}
Here $\varphi^{\,R}_r\colon \bar B_R(p)\to \bar B_r(p)$ is the inducing function defined as in Definition \ref{d:induced subp}.
\end{lemma}

\begin{proof}[Proof of Lemma \ref{t:GH-map on packing-Alex}]

Let us introduce the notation $(\lambda) \bar B_s\equiv(\bar B_s, \lambda d)$ to denote the rescaled space.  The proof consists of two points.  First, we will see that it is almost immediate from the assumed conditions that the mapping $\varphi^R_r: \bar B_{R}(p)\to (R/s)\bar B_s(p)$ is a GH map.  Second, we will show that $\bar B_{R}(p)$ is GH-close to a ball in a cone space $C(\Sigma)$, centered at the cone point. The combination of these two points prove the Lemma.

Let us discuss these points more carefully.  For simplicity, we will only prove the result for $X\in\Alex^n(0)$. The general case is similar with a modifications on $c$, which are just used to estimate the law of cosine formula in (\ref{l:GH-map on packing-Alex.e63}).  Now by the assumptions and the monotonic property, the induced subpacking $\{\varphi_s^R(x_i):x_i\in \bx(p,R)\}$ is a packing of $\bar B_s(p)$ and
  \begin{align}
    R^{-1} d(x_i,x_j)\le s^{-1} d(\varphi^{\,R}_s(x_i),\,\varphi^{\,R}_s(x_j))\le R^{-1} d(x_i,x_j)+\epsilon,
  \end{align}
  for every $s\in[r, R)$ and every $1\le i\neq j\le N$.
  %As an abuse of notation,  we will still denote the composition $\iota_{R/s}\circ\varphi^{\,R}_s\colon \bar B_R(p)\to(R/s) \bar B_s(p)$ by $\varphi^{\,R}_s\colon \bar B_R(p)\to(R/s) \bar B_s(p)$. The actual meaning would follow from the context.

  By Proposition \ref{l:GH-map on packing}, for all $s\in[r, R)$, we therefore have that
  \begin{align}
    d_{GH}\Big(\bar B_R(p),\,(R/s) \bar B_s(p)\Big)\le 4\epsilon R\, ,
    \label{l:GH-map on packing-Alex.e11}
  \end{align}
  where $\varphi^{\,R}_s\colon \bar B_R(p)\to(R/s) \bar B_s(p)$ is a $8\epsilon R$-isometry. To prove (\ref{l:GH-map on packing-Alex.e1}) it therefore suffices to construct a metric cone $C(\Sigma)$ and show that
  \begin{align}
    d_{GH}(\bar B_R(p), \bar B_R(p^*))\le\frac12\epsilon^{\,0.1}R,
    \label{l:GH-map on packing-Alex.e45}
  \end{align}
  where $p^*\in C(\Sigma)$ is the cone point. Let us prove this by first assuming the following lemma, which we will prove later. Let $\iota_{\lambda}\colon \bar B_s\to (\lambda) \bar B_s$ be the identity map.

\begin{lemma} \label{l:prop.varphi} Let $S_\rho=\{x\in \bar B_R(p): d(p,x)=\rho\}$ be the $\rho$-cross section in $X$.
  \begin{enumerate}
  \renewcommand{\labelenumi}{(\roman{enumi})}
  \item For any $t\in[\epsilon^{\,0.5},1)$, the restricted map $\iota_{t^{-1}}\circ\varphi^{\,R}_{tR}|_{S_R}\colon S_R\to (t^{-1})S_{tR}$ is $\epsilon^{\,0.4}R$-onto.
  \item For every $x, y\in S_R$ and $t_1, t_2\in[\epsilon^{\,0.5}, 1)$, we have
    \begin{align}
    \left|\cos\tang{p}{x}{y}-\cos\tang{p}{\varphi_{t_1R}^R(x)}{\varphi_{t_2R}^R(y)}\right|
    &\le\epsilon^{\,0.4}.
    \label{l:GH-map on packing-Alex.e62}
    \end{align}
  \item For any $x, y\in S_R$, geodesic triangle $\triangle pxy$ is $\epsilon^{\,0.3}R$-close to a geodesic triangle in $\dR^2$, equipped with the extrinsic metrics.
%  \item Let $d_{S_R}$ be the induced distance on $S_R$. Then for any $x,y\in S_R$, we have
%    \begin{align}
%    \left|\frac{d_{S_R}(x,y)}{R}-\tang{p}{x}{y}\right|
%    &\le\epsilon^{1/2}.
%    \label{l:GH-map on packing-Alex.e62}
%    \end{align}
\end{enumerate}
\end{lemma}

Using the above we now construct a metric cone $C(\Sigma)$ and define a GH-map $f\colon \bar B_R(p^*)\to \bar B_R(p)$, where $p^*\in C(\Sigma)$ is the cone point. Define a distance function $d_{S_R}$ on the $R$-cross section $S_R=\{x\in X: d(p,x)=R\}$ by
%\textcolor{red}{(Is the following in any way a natural construction?  It seems to allow a large number of errors.  Also, I still dont see a triangle inequality)} \blue{(The notion of intrinsic metric can be defined in the following way (see the attached snap shot from Burago-Burago-Ivanov's book ``A course in Metric Geometry"). By the following Corollary 2.4.17, $d_{S_R}$ can be viewed as a $\epsilon^{\,0.1}R$-approximation to the induced length metric on $S_R$. This kind of approximation was also used in Cheeger-Colding's wraped product paper (3.11). I add an argument below to show the triangle inequality.) }
\begin{align}
  d_{S_R}(x,y)=\inf_{x_0,\,\dots,\,x_N\in S_R}\left\{\sum_{\alpha=1}^Nd_X(w_{\alpha-1},w_\alpha)\colon w_0=x, \,w_N=y, \,d_X(w_{\alpha-1},w_\alpha)\le\epsilon^{\,0.1}R\right\}.
\end{align}
Note that this is an approximation of the induced length space distance function on a subset.  It's clear that $d_{S_R}(x,y)\ge d_X(x,y)$, and thus if $d_{S_R}(x,y)=0$ then $x=y$. To verify the triangle inequality, we let $x,y,z\in S_R$. By the definition we have for any $\eta>0$ that there exists $w_\alpha\in S_R$ with $w_0=x$, $w_{N_1}=y$, $w_{N_2}=z$, $d_X(w_{\alpha-1},w_\alpha)\le\epsilon^{\,0.1}R$, so that $d_{S_R}(x,y)\ge \sum_{\alpha=1}^{N_1}d_X(w_{\alpha-1},w_\alpha)-\eta$ and $d_{S_R}(y,z)\ge \sum_{\alpha=N_1+1}^{N_2}d_X(w_{\alpha-1},w_\alpha)-\eta$. Then we have
\begin{align}
  d_{S_R}(x,y)+d_{S_R}(y,z)\ge \sum_{i=1}^{N_2}d_X(w_{\alpha-1},w_\alpha)-2\eta\ge d_{S_R}(x,z)-2\eta.
\end{align}
Letting $\eta\to 0$ we then obtain the triangle inequality.

Now let $(\Sigma, d_\Sigma)=(S_R,\frac1R d_{S_R})$ and $C(\Sigma)$ be the metric cone over $\Sigma$ and $p^*$ is the cone point. Let $(\tilde S_R,d_{\tilde S_R})=(\Sigma, R\, d_\Sigma)$ be the $R$-cross section in $C(\Sigma)$.
Let $\Pi\colon C(\Sigma)\to \tilde S_R$ by $a=(\bar a, d(p^*,\,a))\mapsto \bar a$ be the projection mapping. Identify $\tilde S_R$ with $S_R$ and let us define
\begin{align}
f\colon\bar B_R(p^*)\to \bar B_R(p) \text{ by } a\mapsto\varphi^{\,R}_{d(p^*,\,a)}\circ\Pi(a)\, .	
\end{align}

We first show that $f$ is $\epsilon^{\,0.4}R$-onto. Let $x\in \bar B_R(p)$. Note that for any $y\in \tilde S_R=S_R$, we have $(y, d_X(p,x))\in C(\Sigma)$ and $f((y, d_X(p,x)))=\varphi^{\,R}_{d(p,x)}(y)$. Thus the $\epsilon^{\,0.4}R$-onto property of $f$ follows from Lemma \ref{l:prop.varphi} (i).

Now we show that $f$ is $\frac12\epsilon^{\,0.1}R$-distance preserving. Let $a,b\in C(\Sigma)$, $x=f(a)$, $y=f(b)$ and $\gamma_{x,y}$ be a geodesic connecting $x$ and $y$. By Lemma \ref{l:prop.varphi} (i), for any partition $\{u_i\}$ of $\gamma_{x,y}$, there exist $w_i\in S_R$ and $s_i>0$ such that $d_X(\varphi^R_{s_iR}(w_i),u_i)\le\epsilon^{\,0.4}R$. Note that for any two points $x',y'\in S_R$ we have that $d_{S_R}(x',y')=d_X(x',y')$ if $d_X(x',y')\le\epsilon^{\,0.1}R$. By Lemma \ref{l:prop.varphi} (iii), the points $u_i$ and $w_i$, $i=1,\dots,N$ can be chosen so that $\frac12\epsilon^{\,0.1}R \ge d_X(w_{i-1}, w_i)=d_{S_R}(w_{i-1}, w_i)\ge \frac14\epsilon^{\,0.1}R$. Thus for this partition we have that $N\le \frac{10R}{\epsilon^{\,0.1}R}=10\epsilon^{-0.1}$.

Now let $\tilde \varphi$ be the inducting function on $C(\Sigma)$ defined in the same way as $\varphi$. By the cone metric, we have
\begin{align}
  d_{C(\Sigma)}\left(\tilde \varphi^R_{s_{i-1}R}(w_{i-1}), \tilde \varphi^R_{s_iR}(w_i)\right)=\sqrt{s_{i-1}s_i\, d_{S_R}^2(w_{i-1},w_i)+(s_{i-1}-s_i)^2R^2}.
\end{align}
By Lemma \ref{l:prop.varphi} (iii), we have
\begin{align}
  \left|\,d_X\left(\varphi^R_{s_{i-1}R}(w_{i-1}), \varphi^R_{s_iR}(w_i)\right)
  -\sqrt{s_{i-1}s_i\, d_X^2(w_{i-1},w_i)+(s_{i-1}-s_i)^2R^2}\,\right|<10\epsilon^{\,0.3}R.
\end{align}
Therefore,
\begin{align}
  d_X(x,y)&=\sum_i d_X(u_{i-1}, u_i)
  \notag \\
  &\ge\sum_{i=1}^N \left(d_X\left(\varphi^R_{s_{i-1}R}(w_{i-1}), \varphi^R_{s_iR}(w_i)\right)-2\epsilon^{\,0.4}R\right)
  \notag \\
  &\ge -12N\cdot\epsilon^{\,0.3}R+\sum_{i=1}^N \sqrt{s_{i-1}s_i\, d_X^2(w_{i-1},w_i)+(s_{i-1}-s_i)^2R^2}
  \notag \\
  &\ge -120\epsilon^{\,0.2}R+\sum_{i=1}^N \sqrt{s_{i-1}s_i\, d_{S_R}^2(w_{i-1},w_i)+(s_{i-1}-s_i)^2R^2}
  \notag \\
  &=-120\epsilon^{\,0.2}R+\sum_{i=1}^N d_{C(\Sigma)}\left(\tilde\varphi^R_{s_{i-1}R}(w_{i-1}), \tilde\varphi^R_{s_iR}(w_i)\right)
  \notag \\
  &\ge-122\epsilon^{\,0.2}R+d_{C(\Sigma)}(a,b).
  \label{l:GH-map on packing-Alex.e81}
\end{align}
The last inequality follows from the triangle inequality since $w_1$ and $w_N$ can be chosen so that $d(a,\tilde\varphi^R_{s_1R}(w_1))<\epsilon^{0.4}R$ and $d(b,\tilde\varphi^R_{s_NR}(w_N))<\epsilon^{0.4}R$.
Starting from a partition of $\gamma_{a,b}$ and apply the same arguments. We get
\begin{align}
  d_X(x,y)\le d_{C(\Sigma)}(a,b)
  +122\epsilon^{\,0.2}R.
  \label{l:GH-map on packing-Alex.e82}
\end{align}
Combining (\ref{l:GH-map on packing-Alex.e81}), (\ref{l:GH-map on packing-Alex.e82}) and the definition of $f$, we get the desired result.
%On the other hand, by a similar construction, there are $b_i\in S_R$ with $d(b_{i-1},b_i)=\frac12\epsilon^{\,0.1}R$ and $s_i$, $c_i=c_i(s_{i-1}, s_i, d(b_{i-1}, b_i))>0$ such that $\tilde\varphi^R_{s_iR}(b_i)\in\gamma_{a,b}$ forms a partition and
%\begin{align}
%  d_{C(\Sigma)}\left(\tilde\varphi^R_{s_{i-1}R}(b_{i-1}), \tilde\varphi^R_{s_iR}(b_i)\right)=c_i d(b_{i-1},b_i).
%\end{align}
%Then we have
%\begin{align}
%  d(a,b)&=\sum_i d_{C(\Sigma)}\left(\tilde\varphi^R_{s_{i-1}R}(b_{i-1}), \tilde\varphi^R_{s_iR}(b_i)\right)
%  \notag \\
%  &= \sum_i c_i d_{S_R}(b_{i-1},b_i)
%  \notag \\
%  &= \sum_i c_i d(b_{i-1},b_i)
%  \notag \\
%  &\ge\sum_i d\left(\varphi^R_{s_{i-1}R}(b_{i-1}), \varphi^R_{s_iR}(b_i)\right)-200\epsilon^{\,0.2}R
%  \notag \\
%  &\ge d(x,y)-200\epsilon^{\,0.2}R.
%  \label{l:GH-map on packing-Alex.e82}
%\end{align}
\end{proof}

Now let us finish the proof of Lemma \ref{l:prop.varphi}:

\begin{proof}[Proof of Lemma \ref{l:prop.varphi}]
  (i) By (\ref{l:GH-map on packing-Alex.e11}), for any $\lambda\in[r/R, 1)\supseteq[1/2,1)$ we have that
  \begin{align}
    \text{$\iota_{\lambda^{-1}}\circ\varphi^{\,R}_{\lambda R}|_{S_\rho}\colon S_\rho\to (\lambda^{-1})S_{\lambda \rho}$ is a $24\epsilon R$-isometry for any $\rho\in(0,R]$.}
    \label{l:GH-map on packing-Alex.e77}
  \end{align}
  This in particular proves (i) for $t\in[r/R, 1)$. For the case $t\ll r/R$, we need to inductively apply $\varphi^{\,R}_{R/2}$.

For any $t\in[\epsilon^{\,0.5}, 1)$, there is an integer $K=K(t)\le\epsilon^{-0.1}$ such that $2^{-(K+1)}\le t< 2^{-K}$. Since in $X$ geodesics do not bifurcate and in the definition of induced packing, the choices of geodesics are a priori fixed in terms of $p$ and $R$, we can write
%  \begin{align}
%    \varphi_{t}=\varphi_{2^{K}t}\circ \underset{K}{\underbrace{\varphi_{1/2}\circ\dots\circ\varphi_{1/2}}}.
%        \label{l:GH-map on packing-Alex.e61}
%  \end{align}
    \begin{align}
    \varphi^{\,R}_{tR}=\varphi^{\,R}_{2^K tR}\circ \underset{K}{\underbrace{\varphi^{\,R}_{R/2}\circ \dots\circ\varphi^{\,R}_{R/2}}}.
        \label{l:GH-map on packing-Alex.e61}
  \end{align}
  Note that $2^Kt\in[1/2, 1)\subseteq[r/R,1)$. Thus (\ref{l:GH-map on packing-Alex.e77}) applies to $\rho=R$ and $\lambda=2^Kt$. Combining (\ref{l:GH-map on packing-Alex.e77}) and (\ref{l:GH-map on packing-Alex.e61}), we get that $\iota_{t^{-1}}\circ\varphi_{tR}^R$ is $24(K+1)\epsilon R$-onto. Then the result follows since $K\le\epsilon^{-0.1}$.

  (ii) We first show that (\ref{l:GH-map on packing-Alex.e62}) is true for $t_1=t_2=t$. Fix $t\in[\epsilon^{\,0.5}, 1)$. Let $x_0=x$, $x_{i}=\varphi^{\,R}_{R/2}(x_{i-1})$ for $1\le i\le K$ and $x_{K+1}=\varphi^{\,R}_{2^K tR}(x_K)=\varphi^{\,R}_{tR}(x)$, where $K=K(t)\le\epsilon^{-0.1}$ is defined as in (i). The sequence $\{y_i\}$ is defined similarly in terms of $y$. By (\ref{l:GH-map on packing-Alex.e77}) and because $2^Kt\in[1/2, 1)$, for $1\le i\le K$ we have
  \begin{align}
    \left|\frac12d(x_{i}, y_{i})-d(x_{i-1}, y_{i-1})\right|\le24\epsilon R
    \label{l:GH-map on packing-Alex.e15}
  \end{align}
  and
    \begin{align}
    \Big|2^Kt\cdot d(x_{K+1}, y_{K+1})-d(x_K, y_K)\Big|\le24\epsilon R.
    \label{l:GH-map on packing-Alex.e16}
  \end{align}
  Note that $d(p, x_{i-1})=d(p, y_{i-1})=\frac12 d(p, x_i)=\frac12 d(p, y_i)$ and $d(p, x_K)=d(p, y_K)=2^Kt\cdot d(p, x_{K+1})=2^Kt\cdot d(p, y_{K+1})$. By law of cosine, we have
  \begin{align}
    \left|\cos\tang{p}{x_i}{y_i}-\cos\tang{p}{x_{i-1}}{y_{i-1}}\right|
    &=\frac12 \left|\left(\frac{d(x_i,\, y_i)}{d(p,x_i)}\right)^2-\left(\frac{d(x_{i-1}, \, y_{i-1})}{d(p,x_{i-1})}\right)^2\right|    \notag\\
    &\le \frac{50\epsilon R}{d(p,x_i)}
    \le \frac{50\epsilon R}{tR}
    \notag\\
    &\le 50\epsilon^{\,0.5}.
    \label{l:GH-map on packing-Alex.e63}
  \end{align}
 Summing up (\ref{l:GH-map on packing-Alex.e63}) for $i=1,2,\dots,K+1$, we get
   \begin{align}
    \left|\cos\tang{p}{x}{y}-\cos\tang{p}{\varphi_{tR}^R(x)}{\varphi_{tR}^R(y)}\right|
    &\le 50(K+1)\epsilon^{\,0.5}\le\epsilon^{\,0.4}.
    \label{l:GH-map on packing-Alex.e64}
  \end{align}
Suppose $t_1\le t_2$. By Topnogov comparison, we have
\begin{align}
  \tang{p}{x}{y}\le\tang{p}{\varphi_{t_1R}^R(x)}{\varphi_{t_2R}^R(y)}\le\tang{p}{\varphi_{t_1R}^R(x)}{\varphi_{t_1R}^R(y)}.
  \label{l:GH-map on packing-Alex.e65}
\end{align}
Then (\ref{l:GH-map on packing-Alex.e62}) follows from (\ref{l:GH-map on packing-Alex.e64}) and (\ref{l:GH-map on packing-Alex.e65}).

The statement (iii) is a direct consequence of (ii).
\end{proof}

Lemma \ref{t:GH-map on packing-Alex} implies that when passing an $\epsilon$-bad scale, either the packing number, or the rescaled distance distortion is increased by at least a definite amount, depending on $\epsilon$.

\begin{corollary}\label{c:packing-GH-bad scale}
  For any $\epsilon>0$, there is $\delta=\delta(\epsilon)>0$ such that the following holds. Let $(X,p)\in\Alex^n(-\delta)$ be an Alexandrov space with $r_{\beta_{\,(k)}}$, $r_{\beta_{\,(k+1)}}$ $\epsilon$-bad scales of $p$. Let $\{x_i\}$ be a maximal $\delta$-packing of $B_{r_{\beta_{\,(k)}}}(p)$ and $\{y_i\}$ be the induced subpacking in $B_{r_{\beta_{\,(k+1)}+1}}(p)$. Then one of the following holds:
  \begin{enumerate}
    \renewcommand{\labelenumi}{(\roman{enumi})}
    \item $P_{\delta}(p,r_{\beta_{\,(k+1)}+1})\ge P_{\delta}(p,r_{\beta_{\,(k)}})+1$;
    \item there exist $i\neq j$ such that $r_{\beta_{\,(k+1)}+1}^{-1} d(y_i,y_{j})> r_{\beta_{\,(k)}}^{-1} d(x_i,x_{j})+\delta$.
\end{enumerate}

\end{corollary}

\begin{proof}
  By the definition of bad scales, we have that $T_p^\epsilon(r_{\beta_{\,(k+1)}+1}, r_{\beta_{\,(k)}})=1$. Then the result follows from Lemma \ref{t:GH-map on packing-Alex} with $\delta=\epsilon^{10}$.
\end{proof}

Now we give a proof of Lemma \ref{l:finite_bad_radii} using the above monotone property.

\begin{proof}[Proof of Lemma \ref{l:finite_bad_radii}] We will only prove for $X\in\Alex^n(0)$ to keep notation simple, the general case is similar. Let $r_\alpha=2^{-\alpha}$, $\alpha\in \dN$ and $K>J\ge 0$ be integers. Let $N_J=P_{\delta}(p,r_J)$ be the maximum $\delta$-packing number of $B_{r_J}(p)$. Let $\mathds I=\{\,\beta_{\,(k)}\in\dN: r_{\beta_{\,(k)}}\in\Bad^\epsilon(p)\cap [r_K,\, r_J]\}$. We claim that if $|\mathds I|>10N_J^2\,\delta^{-1}$, then $P_{\delta}(p,r_K)\ge P_{\delta}(p,r_J)+1$. If this is not true, then $P_{\delta}(p,r_K)=P_{\delta}(p,r_\alpha)=P_{\delta}(p,r_J)$ for every $\alpha\in[J, K]\cap\dZ$. %Thus $\{x_i^{\,\alpha}\}_{i=1}^{N_J}$ is a $\delta/5$-packing for $B_{r_\alpha}(p)$ for every $J\le\alpha\le K$.
Let $\{x_i\}_{i=1}^{N_J}$ be a maximal $\delta$-packing of $B_{r_J}(p)$ and $\{x_i^{\,\alpha}\}$ be the induced subpacking in $B_{r_\alpha}(p)$. By Corollary \ref{c:packing-GH-bad scale}, for every $\beta_{\,(k)}\in\mathds I$, there exist $i$ and $j$, depending on $\beta_{\,(k)}$, such that
  \begin{align}
    r_{\,\beta_{\,(k+1)}+1}^{-1} d(x_i^{\,\beta_{(k+1)}+1},x_j^{\,\beta_{(k+1)}+1}
    )> r_{\,\beta_{(k)}}^{-1} d(x_i^{\,\beta_{(k)}},x_j^{\,\beta_{(k)}})+\delta.
    \label{l:finite_bad_radii.e1}
  \end{align}

  Given a pair of indices $(i,j)$, let $\mathds I_{(i,j)}$ be the collection of $\beta_{\,(k)}\in \mathds I$ such that (\ref{l:finite_bad_radii.e1}) holds. Because $|\,\mathds I\,|>10N_J^2\,\delta^{-1}$, there exist $1\le i_0,j_0\le N_J$, such that $|\,\mathds I_{(i_0,j_0)}\,|> 10\,\delta^{-1}$. Furthermore, there is a subset $\mathds J_{(i_0,j_0)}\subseteq \mathds I_{(i_0,j_0)}$ with $|\mathds J_{(i_0,j_0)}|> 5\,\delta^{-1}$, so that the intervals $\{(\beta_{\,(k)},\, \beta_{\,(k+1)}+1): \,\beta_{\,(k)}\in\mathds J_{(i_0,j_0)}\}$ are disjoint. Note that by the monotonic property, we have
    \begin{align}
    r_{\,\alpha_1}^{-1} d(x_i^{\,\alpha_1},x_j^{\,\alpha_1})\ge r_{\,\alpha_2}^{-1} d(x_i^{\,\alpha_2},x_j^{\,\alpha_2}),
    \label{l:finite_bad_radii.e4}
  \end{align}
  for every $\alpha_1\ge\alpha_2$.
  Summing up (\ref{l:finite_bad_radii.e1}) for $\beta_{\,(k)}\in\mathds J_{(i_0,j_0)}$ and taking in account (\ref{l:finite_bad_radii.e4}),  we get
  \begin{align}
    2\ge r_K^{-1}\,d(x_i^K, x_j^K)> r_J^{-1}d(x_i^J, x_j^J)+\sum_{\beta_{\,(k)}\in \mathds J_{(i_0,j_0)}}\delta\ge 5,
    \label{l:finite_bad_radii.e2}
  \end{align}
  a contradiction.

  Note now that for every $r>0$ we have that $P_{\delta}(p,r)\le C(n,\delta)$. Thus it follows from the above claim that
  \begin{align}
    |\Bad^\epsilon(p)|\le (C(n,\delta)+1)\,(10\,C(n,\delta)^2\delta^{-1}+1).
    \label{l:finite_bad_radii.e3}
  \end{align}
\end{proof}

%\red{The above lemma is also true for extremal subsets in $X$. [Perelman–Petrunin 93, 3.2(3)]. Note that the bi-Lipschitz coefficient depends on the volume, so the effective estimates on the extremal subset will depend on volume as well, if it just comes from the bi-Lipschitz metric.}

For any $\lambda\in(0,1/4)$, it follows from Lemma \ref{l:finite_bad_radii} that for any $x\in X$, there is at least one of the intervals
$$[\lambda^{2\Bad^\epsilon(x)+1},\lambda^{2\Bad^\epsilon(x)}],\, \dots,\, [\lambda^5,\lambda^4], \, [\lambda^3,\lambda^2], \, [\lambda,1]$$
containing no $\epsilon$-bad scale. Thus we have

%$$m=\inf\Big\{\alpha\in\dN:\Bad^\delta(x)\cap[\lambda^{\alpha+1}, \lambda^\alpha] =\varnothing\Big\}\le N(n,\delta).
%$$
%Take $\eta=\lambda^m$. By a straightforward rescaling argument, we have the following lemma.

\begin{lemma}\label{l:uniform_good_scale}
  For any $n\in\dN$, $1/4>\lambda>0$ and $\delta>0$, there exists $\eta=\eta(n,\delta,\lambda)>0$ such that for any $x\in X\in\Alex^n(-1)$ and any $0<R\le 1$, there exists $r_x\ge \eta R$, such that $\Bad^\delta(x)\cap [\lambda r_x, r_x]=\varnothing$ and thus $T^\delta_x(\lambda r_x, r_x)=0$.
\end{lemma}

\vspace{.5cm}

\section{Splitting Theory and Dimension Reduction}\label{s:quant_strat}

%We prove Theorems \ref{t:packing_estimate} and \ref{t:rectifiable} in this section.

\subsection{Splitting theory}\label{subsec: Splitting theory}

In this subsection, we discuss the splitting theory in Alexandrov geometry. Proposition \ref{p:split-theory} is a key geometric property for spaces with lower sectional curvature bounds that distinguishes them from spaces with lesser geometric constraints, such as lower Ricci curvature bounds.  In words, it says that if some ball almost-splits off a Euclidean factor, then all sub-balls continue to almost-split off this factor.

\begin{proposition}\label{p:split-theory}
  For any $n,\epsilon>0$, there exist $\delta=\delta(n,\epsilon)>0$ so that the following holds for any $X\in\Alex^n(-\delta)$ and $R\in(0,1]$.
     \begin{enumerate}
  \renewcommand{\labelenumi}{(\roman{enumi})}
  \item Let $u=(u_1,\dots,u_k)\colon B_{5R}(p)\to \dR^k$ be a $(k,\delta)$-splitting map. For any $B_r\subseteq B_R(p)$ and any $\xi\in u(B_r)$, there exists a map $\phi\colon B_r\to u^{-1}(\xi)$ so that
  $$(u,\phi)\colon B_r\to \dR^k\times u^{-1}(\xi)$$
  is $\epsilon r$-isometry.
  \item If $f\colon B_{5R}(p) \to B_{5R}(z)$, where $z\in\dR^k\times Z$, is a $\delta R$-isometry, then there exists a $(k,\epsilon)$-splitting map $u\colon B_{R}(p)\to\dR^k$.
  %and $\phi\colon B_{\delta R}(p)\to u^{-1}(0)$, such that
   %$$(u,\phi)\colon B_{\delta R}(p)\to \dR^k\times u^{-1}(0)$$
   %is a $\delta^2 R$-isometry.
  \item If there is a $(k,\delta)$-strainer $\{(a_i, b_i)\}$ with $d(p, a_i)$, $d(p, b_i)\ge 5R$ for every $1\le i\le k$, then there exists a $(k,\epsilon)$-splitting map $u\colon B_{R}(p)\to\dR^k$.
  \end{enumerate}
\end{proposition}

%\textcolor{blue}{Just curious, in (ii), it seems that the splitting directions of $u$ are kind of the same as the splitting directions indicated by $f$, but how to describe this? }

%\begin{remark}
 % Note that in (ii), $\pi\circ f$ may not be $\delta^2 R$ close to $u$ on $B_{\delra R}(p)$.
%\end{remark}

The above splitting theory in Alexandrov geometry is well understood. For completeness we outline the proof.

\begin{proof}
We argue $(i)$ by contradiction. This argument can be made effective with some extra work. Note that if $X_i\in\Alex^n(\kappa)$ with $X_i\to X$ and $\delta_i\to\delta$, then the limit of $(k,\delta_i)$-splitting functions on $X_i$ is a $(k,\delta)$-splitting function on $X$. Thus passing to a limit of contradictive rescalled sub-balls, it suffices to show that if $(X,p)\in\Alex^n(0)$ and there is a $0$-splitting function $u=(u_1,\dots,u_k)\colon B_{5}(p)\to \dR^k$, then for any $\xi\in u(B_1(p))$, there exists a map $\phi\colon B_1(p)\to u^{-1}(\xi)$ so that
  $$(u,\phi)\colon B_1(p)\to \dR^k\times u^{-1}(\xi)$$
  is an isometric embedding. For such a $0$-splitting function $u$, the following hold for every $i$ and $j$:
  \begin{enumerate}
  \renewcommand{\labelenumi}{(\arabic{enumi})}
  \item $u_i$ is $0$-concave.
  \item $\langle \nabla u_i,\nabla u_j \rangle=\delta_{ij}$.
  \item For any $x, y\in B_R(p)$ and any minimizing geodesic $\gamma$ connecting $x$ and $y$, it holds that $$\langle \uparrow_x^y, \nabla_x\, u_i \rangle + \langle\uparrow_y^x, \nabla_y\, u_i\rangle=0.$$
  \end{enumerate}
We now prove the result by induction on $k$. Start with the base case $k=1$. Let $\sigma_x(t)$ be a $u$-gradient flow with $\sigma_x(0)=x$. If no confusion arises one may write $x_t=\sigma_x(t)$. Because $u$ is $0$-concave and $|\nabla u|=1$, we have $u(x_t)-u(x)=t$ and $d(u(x_t), u(x_s))=|t-s|$. In particular, $\sigma_x(t)$ is a geodesic from $x$. It's clear that the directed tangent vectors $\sigma^+(t)= \nabla_{x_t}\, u$ and $\sigma^-(t)= -\nabla_{x_t}\, u$.

Let $T_x$ be the time so that $\sigma_x(T_x)\in u^{-1}(\xi)$ and define $\phi(x)=\sigma_x(T_x)\in u^{-1}(\xi)$. We will show that
  $$(u,\phi)\mid_{B_1}\colon B_1\to \dR\times u^{-1}(\xi)$$
  is an isometric embedding. This follows from the following statements for arbitrary $\xi\in u(B_1(p))$ and $t,s\in[0,1]$.
  \begin{enumerate}
  \renewcommand{\labelenumi}{(\Alph{enumi})}
  \item $|T_x|=d(x,u^{-1}(\xi))$.
  \item For any two $u$-gradient curves $\alpha$ and $\beta$, we have $d(\alpha(t),\, \beta(t))=d(\alpha(s),\, \beta(s))$.
  \item The Pythagorean Theorem $d^2(x_t,y)= d^2(x,y)+t^2$.
  \end{enumerate}

We first prove (A). It's clear that $|T_x|\ge d(x,u^{-1}(\xi))$. Recall that if $X\in\Alex^n(0)$ and $f\colon X\to \dR$ is a $\lambda$-concave function, then
\begin{align}
  d(p,q) \cdot\left\langle \uparrow_{p}^{q}, \nabla_{p}\, f\right\rangle
  \ge f(q)-f(p)-\frac\lambda2\cdot d^2(p,q).
\end{align}
Thus if $u\colon X\to \dR$ is a $0$-concave function, then
\begin{align}
  d(p,q)\ge d(p,q) \cdot\left\langle \uparrow_{p}^{q}, \nabla_{p}\, u\right\rangle
  \ge u(q)-u(p).
  \label{p:split-theory.ae5}
\end{align}
Let $y\in u^{-1}(\xi)$ so that $d(x,y)=d(x,u^{-1}(\xi))$, then we have
$$d(x,y)\ge |u(x)-u(y)|=|u(\sigma_x(0))-\xi|=|u(\sigma_x(0))-u(\sigma_x(T_x))|=|T_x|.$$

To prove (B), we let ${x_t}=\alpha(t)$, ${y_t}=\beta(t)$ and $\ell(t)=d(x_t,\, y_t)$. Assume $t\ge s$. Let $\displaystyle\ell^+(t)=\lim_{\eta\to 0^+}\frac{\ell(t+\eta)-\ell(t)}{\eta}$ and $\displaystyle\ell^-(t)=\lim_{\eta\to 0^+}\frac{\ell(t-\eta)-\ell(t)}{\eta}$ be the one-sided derivatives.
  By the first variation formula and because $u$ is $0$-concave, we have
  \begin{align}
    \ell^+(t)
    &\le -\langle \uparrow_{x_t}^{y_t}, \nabla_{x_t}\, u \rangle - \langle\uparrow_{y_t}^{x_t}, \nabla_{y_t}\, u\rangle
\le -\frac{u({y_t})-u({x_t})}{\ell(t)} -\frac{u({x_t})-u({y_t})}{\ell(t)}=0.
    \label{pf:split-theory.e7}
  \end{align}
  Thus we get $\ell(t)\le \ell(s)$. Since $\langle \uparrow_{x_t}^{y_t}, \alpha^+(t) \rangle + \langle\uparrow_{y_t}^{x_t}, \beta^+(t)\rangle=\langle \uparrow_{x_t}^{y_t}, \nabla_{x_t}\, u \rangle + \langle\uparrow_{y_t}^{x_t}, \nabla_{y_t}\, u\rangle\le 0$, we have
%    \begin{align}
%    \langle \uparrow_x^y, \alpha^-(t) \rangle + \langle\uparrow_y^x, \beta^-(t)\rangle\ge -20\delta.
%  \end{align}
%  By the first variation formula again,
    \begin{align}
    \ell^-(t)
    &\le -\langle \uparrow_{x_t}^{y_t}, \alpha^-(t) \rangle - \langle\uparrow_{y_t}^{x_t}, \beta^-(t)\rangle
=\langle \uparrow_{x_t}^{y_t}, \alpha^+(t) \rangle + \langle\uparrow_{y_t}^{x_t}, \beta^+(t)\rangle
    \le 0.
    \label{pf:split-theory.e8}
  \end{align}
  Thus $\ell(t)\ge\ell(s)$.

Now we prove (C). By Toponogov comparison and (\ref{p:split-theory.ae5}), we get that
\begin{align}
  d^2(x_t, y)
  &\le d^2(x,y)+t^2-2t\cdot d(x,y)\cdot\left\langle \uparrow_{x}^{y}, \nabla_{x}\, u\right\rangle
  \notag \\
  &\le d^2(x,y)+t^2-2t\cdot(u(y)-u(x)).
  \label{p:split-theory.ae1}
\end{align}
Start with $x,y\in u^{-1}(\xi)$. Fix $y$ and flow $x$ by time $t$. By (\ref{p:split-theory.ae1}), we get
\begin{align}
  d^2(x_t,y)\le d^2(x,y)+t^2.
  \label{p:split-theory.ae2}
\end{align}
Fix $x_t$ and flow $y$ by time $t$. That is, in (\ref{p:split-theory.ae1}), substitute $y$ by $x_t$, $x$ by $y$ and $x_t$ by $y_t$. We get
\begin{align}
  d^2(y_t,x_t)&\le d^2(y, x_t)+t^2-2t\cdot (u(x_t)-u(y))
  \notag\\
  &\le d^2(x_t,y)+t^2-2t\cdot (u(x_t)-u(x))
  \notag\\
  &= d^2(x_t,y)-t^2.
  \label{p:split-theory.ae3}
\end{align}
Combine (\ref{p:split-theory.ae2}) and (\ref{p:split-theory.ae3}). We have
$$d^2(x_t,y_t)\le d^2(x_t,y)-t^2\le d^2(x,y).$$
By (B), we have $d(x_t,y_t)=d(x,y)$. Thus the Pythagorean Theorem $d^2(x_t,y)= d^2(x,y)+t^2$ follows.

Suppose that the statement has been proved for $k$. Apply the previous argument on the $0$-splitting function $u_{k+1}\colon B_5(p)\to\dR$, we have that $B_1(p)$ is isometric to a ball in $Z\times \dR\in\Alex^n(0)$, and it splits off $\dR^1$ along the direction $\nabla u_{k+1}$. Note that $Z\times \dR\in\Alex^n(0)$ if and only if $Z\in\Alex^{n-1}(0)$. Thus restricted on $Z\times\{0\}\in\Alex^{n-1}(0)$, the map $(u_1,\dots,u_k)$ is $(k,0)$-splitting. Then the result follows from the inductive hypothesis on $k$.

  Assertion (ii) is a consequence of (iii). The proof of (iii) is standard, for instance if $u_i(x)=d(a_i,x)$, then by the arguments used in Sections 5.6 -- 5.7 in \cite{BGP} we have that $u=(u_1,\dots,u_k)$ is a $(k,100\delta)$-splitting function on $B_{\delta R}(p)$, if $\delta=\delta(n)$ is chosen sufficiently small.
%  (ii) By the assumption, there exists $k$-pair of points $\{(a_i, b_i), \, i=1,2\dots,k\}$ such that $d(p,a_i)$, $d(p, b_i)\ge R/2$ for all $i$ and the comparison angles satisfy
%  \begin{align*}
%    \tilde\measuredangle a_ipb_i&>\pi-c\delta, \text{ for all }i;
%    \qquad \tilde\measuredangle a_ipb_j>\frac\pi2-c\delta, \,i\neq j;
%    \\
%    \tilde\measuredangle a_ipa_j&>\frac\pi2-c\delta, \,i\neq j;
%    \qquad\tilde\measuredangle b_ipb_j>\frac\pi2-c\delta, \,i\neq j.
%  \end{align*}
%  Let $u_i(x)=d(a_i,x)$. By the same technique used in subsections 5.6 -- 5.7 in \cite{BGP}. it's not hard to check that $u=(u_1,\dots,u_k)$ is a $(k,c\delta)$-splitting function on $B_R(p)$, if $\delta=\delta(n, R)$ is chosen sufficiently small.
\end{proof}

%For our convenience, the almost Pythagorean theorem is stated below.
%
%\begin{lemma}[Pythagorean theorem]\label{l:phy.thm}
%  There exists a universal constant $c>0$ such that the following holds for every $R\in(0,c]$. Let $u\colon B_{10R}(p)\to \dR^k$ be a $(k, \delta)$-splitting map. Let $x\in B_R(p)$, $z\in u(B_R(p))$ and $\phi\colon B_R(x)\to u^{-1}(z)$ be the function defined as in Lemma \ref{p:split-theory} (i). Let $\bar x=\phi(x)\in u^{-1}(z)$. Then for any $y\in u^{-1}(z)$, we have
%  \begin{align}
%    |d^2(u(x), u(y))+d^2(\bar x, y)-d^2(x,y)|<c\delta^2 d^2(x,y).
%  \end{align}
%  In particular, for $y=\bar x$, we have $|d(u(x), u(\bar x))- d(x, \bar x)|< c\delta d(x, \bar x)$.
%\end{lemma}

%\begin{proof} Not losing generality, assume $r=d(x,y)\le R$. Note that
%  \begin{align}
%    (u,\phi)(x)&=(u(x), \phi(x))=(u(x), \bar x),
%    \notag \\
%    (u,\phi)(y)&=(u(y), \phi(y))=(u(y),y).
%  \end{align}
%  Then the result follows since $(u,\phi)|_{B_r(x)}$ is a $c\delta r$-isometry.
%\end{proof}

The following statement is an easy consequence of Proposition \ref{p:split-theory}.
\begin{corollary}\label{c:sp in all scale}
  For any $n, k\in\dN$ and $\epsilon>0$, there exists $\delta=\delta(n,\epsilon)>0$ so that if $X\in\Alex^n(-\delta)$ and $B_5(p)$ is $(k,\delta)$-splitting, then $B_r(x)$ is $(k, \epsilon)$-splitting for every $x\in B_1(p)$ and every $r\in(0, 1]$.
\end{corollary}

\vspace{.3cm}

\subsection{Strong and weak singularity}\label{subsection:Strong and weak singularity}

In this Subsection we discuss the relations between the strong and weak quantitative singular sets. In fact, they are equivalent in some sense for Alexandrov spaces.

Define weak singular sets
\begin{align}
  \widetilde\cS^{\,k}_{\epsilon,\,r}(X)
  =\{x\in X: B_s(x)\text{ is not $(k+1,\epsilon)$-splitting for every $s\in(r,1]$} \}.
  \label{d:WS}
\end{align}
It's clear that $\widetilde\cS^{\,k}_{\epsilon,\,r}(X)\subseteq\cS^k_{\epsilon,\,r}(X)$.

By Corollary \ref{c:sp in all scale}, we have
\begin{proposition} \label{p:sing.equv}
   For any $n,\epsilon>0$, there exists $\delta(n,\epsilon)>0$ such that for any $X\in\Alex^n(-\delta)$ and $0<r\le 1$, we have
    \begin{align}
     \widetilde\cS^{\,k}_{\epsilon,\,r}(X)\subseteq\cS^k_{\epsilon,\,r}(X) \subseteq  \widetilde\cS^{\,k}_{\delta, r}(X).
   \end{align}
 \end{proposition}

% As a summary, we have
% \begin{align}
%   \widetilde\cS^{\,k}_{\epsilon,\,r}\subseteq\cS^k_{\epsilon,\,r}
%   \subseteq  \widetilde\cS^{\,k}_{\delta, 5r}\subseteq\cS^k_{\delta, 5r}.
%   \label{sing.eqv.s1}
% \end{align}

The quantitative singular sets defined for the Ricci cases in \cite{CN13} is as follows.  Note that we do not use it in this paper and it may be skipped, we are presenting this for comparison sake to the Ricci curvature context.

 \begin{definition}[Quantitative symmetric]\quad

\begin{enumerate}
\item Given a metric space $Y$ and $k\in\dN$, we say that $Y$ is $k$-symmetric if $Y\equiv \dR^k\times C(\Sigma)$ for some metric space $\Sigma$.
\item Given $x\in X$ we say that $B_r(x)$ is $(k,\epsilon)$-symmetric if there exists a $k$-symmetric space $Y$ such that $d_{GH}(B_r(x),B_r(y))\le\epsilon r$, where $y\in Y$ is a cone point.
\end{enumerate}
\end{definition}

Define
\begin{align}
  \cW\cS^{\,k}_{\epsilon,\,r}(X) \equiv \{x\in X: B_s(x)\text{ is not $(k+1,\epsilon)$-symmetric, for every $s\in(r, 1]$} \}.
  \label{d:qs-ric.e1}
 \end{align}
It's clear that $\widetilde\cS^{\,k}_{\epsilon,\,r}(X)\subseteq \cW\cS^{\,k}_{\epsilon,\,r}(X)$.

The following is an easy lemma, by a standard contradiction argument.

\begin{lemma}\label{l:k+0=k_splitting}
For each $n\in\dN$ and $\epsilon>0$ there exists $\delta(n,\epsilon)>0$ such that the following holds for any metric space $(X,p)$. If $B_r(p)$ is both $(0,\delta)$-symmetric and $(k,\delta)$-splitting, then $B_r(p)$ is $(k,\epsilon)$-symmetric.
\end{lemma}

\begin{proposition}
  For any $\epsilon>0$, there exist $\eta(n,\epsilon)$ and $\delta(n,\epsilon)>0$ such that for any $X\in\Alex^n(-\delta)$ and $0<r\le 1$, we have
\begin{align}
  \cW\cS^{\,k}_{\epsilon,\,\eta r}(X)\subseteq \widetilde\cS^{\,k}_{\delta,r}(X)
  \subseteq \cW\cS^{\,k}_{\delta,r}(X).
\end{align}
\end{proposition}

\begin{proof}
  If $x\notin\widetilde\cS^{\,k}_{\delta,r}(X)$, then $B_s(x)$ is $(k+1,\delta)$-splitting for some $s\ge r$. By Corollary \ref{c:sp in all scale}, we have that $B_t(x)$ is $(k+1,\delta_1)$-splitting for all $t\in(0,\frac15 r]$. On the other hand, by Lemma \ref{l:uniform_good_scale}, there exists $\eta(n,\delta_1)>0$ and $r_x\in[\eta r, \frac15 r]$ such that $B_{r_x}(x)$ is $(0,\delta_1)$-symmetric. Due to Lemma \ref{l:k+0=k_splitting}, with appropriately selected $\delta$ and $\delta_1$, we have that $B_{r_x}(x)$ is $(k+1,\epsilon)$-symmetric. Therefore, $x\notin \cW\cS^{\,k}_{\epsilon,\,\eta r}(X)$.
\end{proof}

\begin{remark}
Our notion of quantitative splitting for Alexandrov spaces is also equivalent to those defined using strainers. In particular, there exists $0<\delta_1(n,\epsilon)<\delta_2(n,\epsilon)$ so that
\begin{align}
\cS^{\,k}_{\epsilon,\,r/5}(X)\subseteq \{x\in X: x \text{ does not admit any } (k+1,\delta_2)\text{-strainer with size } \ge r\}\subseteq \cS^{\,k}_{\delta_1,r}(X)
\end{align}
for any $X\in\Alex^n(-\delta_1)$ and $0<r\le 1$.
\end{remark}

\begin{remark}
By a similar argument, one can show that if $X$ is a $v$-non-collapsed limit of $n$-dimensional manifolds with $\Ric\ge -1$, then there exist $\eta_i(n,\epsilon,v)>0$, $i=1,2$, such that
$$\cW\cS^{\,k}_{\epsilon,\,\eta_1 r}(X)\subseteq \widetilde\cS^{\,k}_{\eta_2,r}(X)\subseteq \cS^{\,k}_{\eta_2,r}(X).$$
However, the statement in the form $\cS^{\,k}_{\epsilon,\,\eta_1 r}(X)\subseteq \cW\cS^{\,k}_{\eta_2,\, r}(X)$ doesn't hold for the Ricci case. %Given a Ricci curvature lower bound, for which if $B_5(p)$ is $(k,\delta)$-splitting, then $B_r(x)$ is $(k,\epsilon)$-splitting for most $x\in B_1(p)$, but not every $x\in B_1(p)$.
\end{remark}

\vspace{.3cm}

\subsection{Dimension reduction}

Note in a metric cone $C(\Sigma)$ that the tangent cone at any point $p\in C(\Sigma)$ away  from the cone point splits off an extra $\dR$-factor in comparison to $C(\Sigma)$.  This is the basis of Federer dimension reduction.  The following lemma is a quantitative version of this on Alexandrov spaces.

\begin{lemma}\label{l:cone split}
  For any $n, k\in\dN$ and $\epsilon>0$, there exists $\delta=\delta(n,\epsilon)$ and $\beta =\beta (n,\epsilon)>0$ such that the following holds for any $(X,p)\in\Alex^n(-\delta)$ and $(k,\delta)$-splitting function $u=(u_1,\dots,u_k)\colon B_{50}(p)\to \dR^k$.
Let $x\in B_1(p)$ and $y\in X$ with $d(x,y)=r>0$.
\begin{enumerate}
  \renewcommand{\labelenumi}{(\roman{enumi})}
  \item  If $T_x^\delta(r,2r)=0$ and $d(x,y)-d(u(x),u(y))> \epsilon r$, then $B_{s}(y)$ is $(k+1,\epsilon)$-splitting for every $0<s\le \beta  r$.
  \item If $T_x^\delta(r,2r)=0$ and $B_{s}(y)$ is not $(k+1,\epsilon)$-splitting for some $0<s\le \beta  r$, then
    \begin{align}
    \big|d(u(x),u(y))-d(x,y)\big|\le \epsilon\, d(x,y).
    \label{e:split on good scale-e2}
  \end{align}
\end{enumerate}
\end{lemma}

\begin{proof}
  We only need to prove (i) since it is equivalent to (ii), taking in account that $|\nabla u|<1+\delta$. Let $\delta=\delta(n,\epsilon)$, $\delta_i=\delta_i(n,\epsilon)$ be constants with $0<\delta<\delta_1<\delta_2<\dots<\epsilon$.

    \begin{figure}[!h]
  \includegraphics[scale=1]{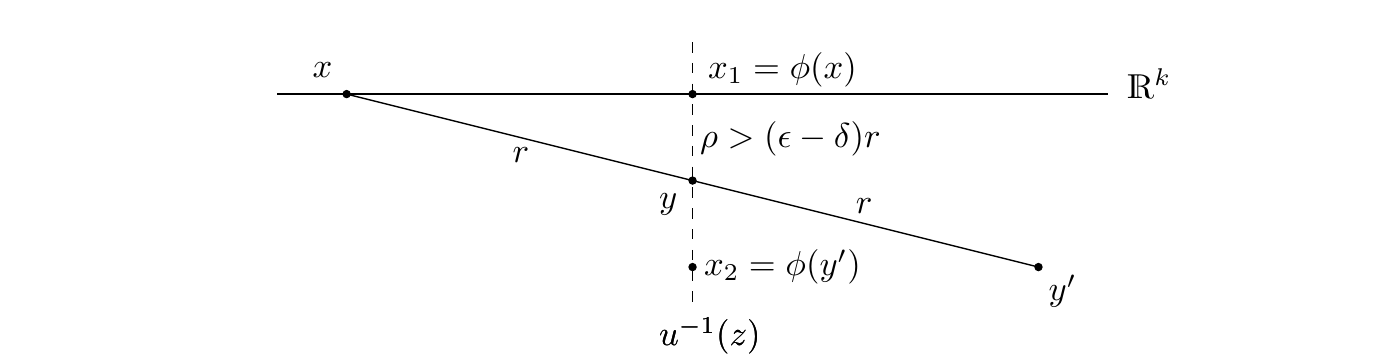}
  \caption{}
  \end{figure}

Let us take $z\equiv u(y)$. Choosing $\delta(n,\epsilon)>0$ small we have by Proposition \ref{p:split-theory} that there exists $\phi\colon B_{10r}(x)\to u^{-1}(z)$ such that $(u,\phi)\colon B_{10r}(x)\to \dR^k\times u^{-1}(z)$ is a $\delta_1 r$-isometry. Let $x_1=\phi(x)\in u^{-1}(z)$ and $\rho=d(x_1, y)$.

We first find a splitting function along the slice $u^{-1}(z)$, using that $y$ is away from the cone point $x_1$. Because $T_x^\delta(r,2r)=0$, there exists $y'\in X$ so that \begin{align}
 |d(y,y')-r|\le \delta_1 r,\qquad |d(x,y')-2r|\le \delta_1 r.
  \label{e:split on good scale-e3}
\end{align}
Let $x_2=\phi(y')\in u^{-1}(z)$. Combine (\ref{e:split on good scale-e3}), $d(x,y)=r$ and the $\delta r$-splitting structure on $B_{10r}(x)$. We have
\begin{align}\label{e:strainer point good scale}
  |d(x_2,y)-\rho|\le4\delta_1 r,
  \qquad |d(x_1,x_2)-2\rho|\le4\delta_1 r.
\end{align}
To see the maximal scales that $B_s(y)$ splits, we need a lower bound of $\rho$. Because $(u,\phi)$ is a $\delta_1 r$-isometry and by the assumptions, we have
\begin{align}
  \rho=d(x_1,y)
  &\ge d(x,y)-d(x, x_1)
  \notag \\
  &\ge d(x,y)-d\Big((u, \phi)(x),\, (u, \phi)(x_1) \Big)-\delta_1 r
  \notag \\
  &=d(x,y)-d\Big((u(x), \phi(x)),\, (u(y), \phi(x)) \Big)-\delta_1 r
  \notag \\
  &=d(x,y)-d\Big(u(x),\, u(y)\Big)-\delta_1 r
  \notag \\
  &>(\epsilon-\delta_1)r.
\end{align}
Choosing $\delta_1$ small and by \eqref{e:strainer point good scale}, we have that $\{x_1,x_2\}$ forms a $(1,\delta_2)$-strainer. Thus $u_{k+1}(q)\equiv d(q,x_1)$ is a $(1,\delta_3)$-splitting map on $B_{\frac15\epsilon r}(y)$.

By the $\delta_1$-almost splitting structure of $B_{10}(p)$ and the Toponogov comparison using \eqref{e:strainer point good scale}, we have
  \begin{align}
    |\langle\nabla_q u_i, \nabla_q u_{k+1}\rangle|<\delta_3.
  \end{align}
  for any $q\in B_{\frac15\epsilon r}(y)$ and every $i=1,2,\dots k$. Thus the function $(u, u_{k+1})|_{B_{\frac15\epsilon r}(y)}$ is a $(k+1,\delta_3)$-splitting map. By Proposition \ref{p:split-theory} (i), $B_{s}(y)$ is $\epsilon$-splitting for every $0<s\le\frac1{50}\epsilon r$.
\end{proof}

Using Lemma \ref{l:cone split}, we can prove the rectifiability.

\begin{proof}[Proof of Theorem \ref{t:rectifiable}]
  Note that $(\cS^k_\epsilon\setminus\cS^{k-1}_\delta)\cap B_1$ can be covered by countably many balls $\{B_{\delta r_i}(x_i)\}$ with $x_i\notin \cS^k_{\delta, 50 r_i}$. That is, $B_{50 r_i}(x_i)$ is $(k,\delta)$-splitting. By Proposition \ref{p:split-theory}, for each of $B_{\delta r_i}(x_i)$, there exists a $\delta_1$-splitting map $u_i\colon B_{50\delta r_i}(x_i)\to\dR^k$. Note that $\dim_{\cH}(\cS^{k-1}_\delta)\le k-1$. Thus it suffices to prove the following statement. There exists $\delta=\delta(n,\epsilon)>0$, such that if there is a $(k,\delta)$-splitting map $u\colon B_{50}(p)\to\dR^k$, then $\cS^k_{\epsilon}\cap B_1(p)$ is $k$-rectifiable.

Let $\delta(n,\epsilon)>0$ be determined later. Recall that by Lemma \ref{l:finite_bad_radii}, for every point $x\in X$, the number of $\delta$-bad scales is at most $N(n,\delta)$. For each $x\in B_1(p)$, let $s_x\in(0,1]$ be the minimum of $1$ and the smallest $\delta$-bad scale at $x$. Put $\Gamma_x^{\,t}=B_{t/2}(x)\cap \cS^k_\epsilon \cap \{y\in B_1(p):s_y>2t\}$. We claim that for any $t>0$, the map $ u|_{\Gamma_x^{\,t}}\colon {\Gamma_x^{\,t}}\to\dR^k$ is bi-Lipschitz onto its image. Once the claim is proved, we have that $\cS^k_\epsilon \cap \{y\in B_1(p):s_y>2t\}=\cup_{x\in B_1(p)} \Gamma_x^{\,t}$ is $k$-rectifiable. Therefore,
  $$\cS^k_\epsilon\cap B_1(p)=\bigcup_{t>0}\left(\cS^k_\epsilon\cap \{y\in B_1(p):s_y>2t\}\right)$$
  is rectifiable.

  Now we prove the claim. Let $x_1, y_1\in \Gamma_x^{\,s}$. Then $d(x_1, y_1)\le t<s_{x_1}/2$. Because $\Bad^\delta(x_1)\cap[0,s_{x_1})=\varnothing$, by Proposition \ref{p:interval no bad}, we have $T_{x_1}^\delta(0,2t)=T_{x_1}^\delta(0,s_{x_1})=0$. Note that $y_1\in \cS^k_\epsilon$ and thus $B_\rho(y_1)$ is not $(k+1,\epsilon)$-splitting for every $\rho\in(0,1]$. By Lemma \ref{l:cone split} (ii), we have
  \begin{align}
    \big|d( u(x_1), u(y_1))-d(x_1,y_1)\big|\le \epsilon\, d(x_1,y_1).
  \end{align}
\end{proof}

%
%Note that for every $x\in X$, there exists $r=r(x)>0$ such that $[0,2r]\cap \Bad^\delta(x)=\varnothing$ \textcolor{red}{(This is a comment which should also be in the last proof)}. Thus Lemma \ref{l:cone split} (ii) also implies that
%
%\begin{corollary}\label{c:restrict GH bilip}
%  There exists $\delta=\delta(n,\epsilon)>0$ such that the following holds. For every $x\in \cS^k_\delta\setminus \cS^{k-1}_\delta$, there exists $r=r(x)>0$ such that the restricted $(k,\epsilon)$-splitting map $\phi|_{B_r(x)\cap \cS^k_\epsilon\setminus \cS^{k-1}_\epsilon}$ is $(1+20\epsilon)$-bi-Lipschitz at $x$. \textcolor{red}{(Is this $r(x)$ quite the same as the one above?  Maybe within a factor?)}
%\end{corollary}

\vspace{.3cm}

\section{Packing Estimates}

We prove Theorem \ref{t:packing_estimate} in this section. The following is the key lemma. %Using Covering Theorem \ref{l:good_scale_covering}, one can easily prove that under the assumption of Lemma \ref{l:cone split}, we have

\begin{lemma}\label{l:finite non-sym scale-1} For any $n\in\dN$ and $\epsilon>0$, there exist $\delta(n,\epsilon)>0$ and $\beta (n,\epsilon)>0$ so that the following holds for any $(X,p)\in\Alex^n(-\delta)$. Suppose $u\colon B_{50}(p)\to \dR^k$ is a $(k,\delta)$-splitting function, and let $\{B_{r_i}(x_i)\}$ with $i\in\mathds I$ be a disjoint collection of balls living on a fixed level set $x_i\in u^{-1}(z)\cap B_1(p)$ for some $z\in\dR^k$.  Then if $x_i\in\cS^k_{\epsilon,\,\beta  r_i}$ we have the estimate $|\mathds I|<N(n,\epsilon)$.
%\begin{align}
%  \Big|\Big\{i: x_i\in\cS^k_{\epsilon,\,\beta  r_i}\cap u^{-1}(z)\cap B_1(p)\Big\}\Big|<N(n,\epsilon).
%\end{align}% $B_{\beta  r_i}(x_i)$ is not $(k+1,\epsilon)$-splitting
\end{lemma}

\begin{proof} %Let $\mathds I=\Big\{i: x_i\in\cS^k_{\epsilon,\,\beta  r_i}\cap u^{-1}(z)\cap B_1(p)\Big\}$.
We will construct a sequence Vitali coverings of $u^{-1}(z)\cap B_1(p)$, which ``converges" to  $\{B_{r_i}(x_i),\, i\in \mathds I\}$. The constants $\delta(n,\epsilon)$, $\eta(n,\epsilon)>0$ and $\lambda>0$ will be determined later.

(Step 1.) Let $\bar B_\rho$ be an arbitrary closed ball with $W\subseteq \bar B_\rho$ be a closed subset and $\mathds I(W)=\{i\in \mathds I: x_i\in W\}$. For $x\in W$ and $0<\epsilon, s\le 1$, define function
   \begin{align}
   \sigma(x,\epsilon, s) &= \renewcommand{\arraystretch}{1.5}
   \left\{\begin{array}{@{}l@{\quad}l@{}}
    \inf\Big\{\tau: T^\epsilon_{x}(\tau s, 2s)=0\Big\}, & \text{ if \quad} T^\epsilon_{x}(s, 2s)=0;
    \\
    1, & \text{ otherwise.}
  \end{array}\right.
  \label{l:good_scale_covering.e5}
\end{align}
By Lemma \ref{l:uniform_good_scale}, for each $0<\lambda<4^{-1}$ there exists $\eta=\eta(n,\epsilon, \lambda)>0$ such that for any $x\in W$, there exists $r_x\in[\eta\rho,\,\rho]$ such that $T^\epsilon_{x}(\lambda r_x, 2r_x)=0$. Therefore, we have
\begin{align}
  \lambda_{x}\equiv\sigma(x,\epsilon,\,r_x)\le \lambda
  \label{l:good_scale_covering.e6}
\end{align}

Define $\mathcal F(W)=\{i\in \mathds I(W): r_i\ge \frac{1}{10}\eta\rho\}$ and ${\mathcal F^c}(W)=\mathds I(W)\setminus \mathcal F(W)$. It is clear that $|\mathcal F(W)|\le N(n,\eta)$, since $r_i\ge \frac{1}{10}\eta\rho$ and $W\subseteq \bar B_\rho$.  Now because $r_i<\frac{1}{10}\eta\rho\le\frac1{10}r_{x_i}$ for every $i\in {\mathcal F^c}(W)$, we have that $\{B_{\frac1{10}r_{x_i}}(x_i), i\in{\mathcal F^c}(W)\}$ is a covering of $\cup_{i\in {\mathcal F^c}(W)}B_{r_i}(x_i)$. Let $\mathcal G(W)\subseteq {\mathcal F^c}(W)$ be a collection of indices so that $\{B_{\frac1{10}r_{x_j}}(x_j),\, j\in\mathcal G(W)\}$ covers $\cup_{i\in {\mathcal F^c}(W)}B_{r_i}(x_i)$, while $\{B_{\frac1{50}r_{x_j}}(x_j),\, j\in\mathcal G(W)\}$ are disjoint. It's clear that $|\mathcal G(W)|\le N(n,\eta)$, since $r_{x_i}\ge \eta\rho$. Now we have
\begin{align}
  \mathds I(W)
  &\subseteq\mathcal F(W)
  \cup \left(\underset{j\in\mathcal G(W)}\cup\mathds I\left(B_{\frac1{10}r_{x_j}}(x_j)\right)\right),
    \label{e:finite non-sym scale-2}
\end{align}
where $|\mathcal F(W)|+|\mathcal G(W)|\le N(n,\eta)$.

Note that function $\sigma(x,\epsilon, s)$ is semi-continuous in $x$. That is, $\displaystyle\liminf_{z\to y}\sigma(z,\epsilon,\,r_x)\ge\sigma(y,\epsilon,\,r_x)$. For each $j\in \mathcal G(W)$, there exists $y_j\in \bar B_{\lambda_{x_j}r_{x_j}}(x_j)\cap W$ so that \begin{align}
  \sigma_{y_j}\equiv\sigma(y_j,\epsilon,\,r_{x_j})=\inf\{\sigma(x,\epsilon,\,r_{x_j}): x\in \bar B_{\lambda_{x_j}r_{x_j}}(x_j)\}\le\lambda_{x_j}\le\lambda.
\end{align}
We claim that $\mathcal F(W)$ and $\mathcal G(W)$ satisfy the following properties.
\begin{enumerate}
  \renewcommand{\labelenumi}{(\arabic{enumi})}
  \setlength{\itemsep}{1pt}
  \item $|\mathcal F(W)|+|\mathcal G(W)|\le N(n,\eta)$.
  \item $\displaystyle\mathds I(W)=\mathcal F(W)
  \cup \left(\underset{j\in\mathcal G(W)}\cup\mathds I\left(B_{\lambda_{x_j}r_{x_j}}(x_j)\cap B_{\sigma_{y_j}r_{x_j}}(y_j)\right)\right)$.
  \item  If $\sigma_{y_j}>0$, then for every $z\in \bar B_{\lambda_{x_j}r_{x_j}}(x_j)\cap \bar B_{\sigma_{y_j}r_{x_j}}(y_j)$
  \begin{align}
  \Big|\Bad^\epsilon(z)\cap[\sigma_{y_j}r_{x_j},1]\Big|
  \ge
  \Big|\Bad^\epsilon(z)\cap[\rho,1]\Big|+1.
  \end{align}

\end{enumerate}
Statement (1) has been proved in the construction. To prove (2), we start with an obvious inclusion formula:
\begin{align}
  B_{\frac1{10}r_{x_j}}(x_j)
  &\subseteq\left(A_{\lambda_{x_j}r_{x_j}}^{\frac1{10}r_{x_j}}(x_j)
  \cup A_{\sigma_{y_j}r_{x_j}}^{\frac1{10}r_{x_j}}(y_j)\right)
  \cup \Big(\bar B_{\lambda_{x_j}r_{x_j}}(x_j)\cap\bar B_{\sigma_{y_j}r_{x_j}}(y_j)\Big)
  \notag\\
  &\cup \Big(\bar B_{\lambda_{x_j}r_{x_j}}(x_j)\setminus \bar B_{r_{x_j}}(y_j)\Big)
  \cup \Big(\bar B_{\sigma_{x_j}r_{x_j}}(y_j)\setminus \bar B_{r_{x_j}}(x_j)\Big).
  \label{e:finite non-sym scale-3}
\end{align}
Let $\lambda<\frac1{10}$ be a constant. Note that $d(x_j, y_j)\le \sigma_{y_j}r_{x_j}\le \lambda_{x_j}r_{x_j}\le r_{x_j}/10$. Thus we have
$\bar B_{\lambda_{x_j}r_{x_j}}(x_j)\subseteq \bar B_{r_{x_j}}(y_j)$ and $\bar B_{\sigma_{y_j}r_{x_j}}(y_j)\subseteq \bar B_{r_{x_j}}(x_j)$.  In particular we then have the better inclusion
\begin{align}
  B_{\frac1{10}r_{x_j}}(x_j)
  &\subseteq\left(A_{\lambda_{x_j}r_{x_j}}^{\frac1{10}r_{x_j}}(x_j)
  \cup A_{\sigma_{y_j}r_{x_j}}^{\frac1{10}r_{x_j}}(y_j)\right)
  \cup \Big(\bar B_{\lambda_{x_j}r_{x_j}}(x_j)\cap\bar B_{\sigma_{y_j}r_{x_j}}(y_j)\Big)\, .
  \label{e:finite non-sym scale-3-new}
\end{align}
It remains to show that
\begin{align}
\mathds I\left(A_{\lambda_{x_j}r_{x_j}}^{\frac1{10}r_{x_j}}(x_j)\right)=\mathds I\left(A_{\sigma_{y_j}r_{x_j}}^{\frac1{10}r_{x_j}}(y_j)\right)=\varnothing.
\end{align}

Suppose $\mathds I\left(A_{\lambda_{x_j}r_{x_j}}^{\frac1{10}r_{x_j}}(x_j)\right)\neq\varnothing$. That is, there exists $i\in {\mathcal F^c}(W)$, so that $\lambda_{x_j}r_{x_j}\le d(x_i,x_j)\le \frac1{10}r_{x_j}$.
 Then from the definition of $r_{x_j}$ we have
 \begin{align}
 T_{x_j}^{\epsilon}\Big(d(x_i, x_j),\, 2d(x_i, x_j)\Big)=0\, .	
 \end{align}
Now let $\beta =\beta (n,\epsilon)>0$ be the constant determined in Lemma \ref{l:cone split}. Because $B_{\beta r_i}(x_i)$ is not $(k+1,\epsilon)$-splitting and $r_i\le d(x_i,x_j)$, the restricted  map $u|_{\{x_i,x_j\}}$ is $(1\pm 4\epsilon)$-bi-Lipschitz. This contradicts to the assumption $u(x_i)=u(x_j)=z$. The proof for $\mathds I\left(A_{\sigma_{y_j}r_{x_j}}^{\frac1{10}r_{x_j}}(y_j)\right)=\varnothing$ is similar.

To prove (3), let $z\in \bar B_{\lambda_{x_j}r_{x_j}}(x_j)\cap \bar B_{\sigma_{y_j}r_{x_j}}(y_j)\in \mathcal D(W_r)$. By the definition of $\sigma_{y_j}$, we have $\sigma(z,\epsilon,\,r_{x_j})\ge\sigma_{y_j}>0$. Thus $T_z^\epsilon\left(\frac12\sigma_{y_j}r_{x_j}, r_{x_j}\right)=1$. By the definition of bad scales, this implies $\Big|\Bad^\epsilon(z)\cap[\sigma_{y_j}r_{x_j},r_{x_j}]\Big|\ge 1$. Then (3) follows since $[\sigma_{y_j}r_{x_j},r_{x_j}]\subseteq[\sigma_{y_j}r_{x_j},\rho]$.

(Step 2.) In this step we construct a covering of $\mathds I$ inductively. Let the decomposition functions $\mathcal F$ and $\mathcal G$ be defined in Step 1. Begin with $W=B_1(p)$. Let $\mathcal C_1=\mathcal F(W)$ and $\mathcal D_1=\mathcal G(W)$. Suppose $\mathcal C_k$ and $\mathcal D_k$ have been constructed and satisfy the following $(A_k) - (C_k)$:
\begin{enumerate}
  \renewcommand{\labelenumi}{($\Alph{enumi}_k$)}
  \setlength{\itemsep}{1pt}
  \item $|\mathcal C_k|\le kN(n,\eta)^k$, $|\mathcal D_k|\le N(n,\eta)^k$.
  \item $\displaystyle\mathds I=\mathcal C_k
  \cup \left(\underset{j\in\mathcal D_k}\cup\mathds I\left(B_{\lambda_{x_j}r_{x_j}}(x_j)\cap B_{\sigma_{y_j}r_{x_j}}(y_j)\right)\right)$.
  \item  $\big|\Bad^\epsilon(z)\cap[\sigma_{y_j}r_{x_j},1]\big|\ge k$ for any $j\in\mathcal D_k$ and $z\in B_{\lambda_{x_j}r_{x_j}}(x_j)\cap B_{\sigma_{y_j}r_{x_j}}(y_j)$, provided $\sigma_{y_j}>0$.
\end{enumerate}
For each $j\in\mathcal D_k$ and $W_j=\bar B_{\lambda_{x_j}r_{x_j}}(x_j)\cap \bar B_{\sigma_{y_j}r_{x_j}}(y_j)$, using the construction of Step 1 let
$$\mathcal C_{k+1}=\mathcal C_k\cup\left(\cup_{j\in\mathcal D_k}\mathcal F(W_j)\right)$$
and
$$\mathcal D_{k+1}=\cup_{j\in\mathcal D_k}\mathcal G(W_j).$$

Now we prove $(A_{k+1})$ -- $(C_{k+1})$ for $\mathcal C_{k+1}$ and $\mathcal D_{k+1}$. By (1) in Step 1, we have $|\mathcal F(W_j)|+|\mathcal G(W_j)|\le N(n,\eta)$. Thus
$$|\mathcal C_{k+1}|\le |\mathcal C_k|+N|\mathcal D_k|\le kN^k+N^{k+1}\le (k+1)N^{k+1}$$
and
$$|\mathcal D_{k+1}|\le N|\mathcal D_k|\le N^{k+1}.$$
Statements $(B_{k+1})$ and $(C_{k+1})$ follow from (2) and (3) respectively.

(Step 3.) By Lemma \ref{l:finite_bad_radii}, the number of $\epsilon$-bad scales is at most $K=K(n,\epsilon)$. Thus due to $(C_k)$, we have $\mathcal D_k=\varnothing$ if $k>K$. Therefore, $\mathds I=\mathcal C_{K}$ and $|\mathds I|=|\mathcal C_{K}|\le KN^{K}$.
\end{proof}

Furthermore, we have the following theorem.

\begin{theorem}\label{t:finite non-sym scale} For any $n\in\dN$, $\epsilon>0$ and $\Lambda\ge 1$, there exist $\delta(n,\epsilon)>0$ and $\beta (n,\epsilon)>0$ so that the following holds for any $(X,p)\in\Alex^n(-1)$. Suppose that there is a $(k,\delta)$-splitting function $u\colon B_{50}(p)\to \dR^k$. If $\{B_{r_i}(x_i)\}$ are disjoint and $B_{\beta \Lambda r_i}(x_i)\cap\cS^k_{\epsilon,\,\beta \Lambda r_i}\neq\varnothing$ for all $i\in\mathds I$, then for any $z\in\dR^k$, we have
\begin{align}
  \Big|\Big\{i\in\mathds I\colon B_{\beta \Lambda r_i}(x_i)\cap u^{-1}(z)\neq \varnothing\Big\}\Big|<N(n,\epsilon, \Lambda).
  \label{t:finite non-sym scale.se1-0}
\end{align}
Additionally, if $r_i=r$ with $B_{\Lambda r}(x_i)\cap\cS^k_{\epsilon,\,\Lambda r}\neq\varnothing$ and $\{B_{r}(x_i)\}$ are disjoint for all $i\in\mathds I$, then for any $z\in\dR^k$, we have
\begin{align}
  \Big|\Big\{i\in\mathds I\colon B_{\Lambda r}(x_i)\cap u^{-1}(z)\neq \varnothing\Big\}\Big|<N(n,\epsilon, \Lambda).
  \label{t:finite non-sym scale.se2-0}
\end{align}
\end{theorem}

%\begin{corollary}
% For any $n\in\dN$, $\epsilon>0$ and $\Lambda\ge 1$, there exist $\delta(n,\epsilon)>0$ and $\beta (n,\epsilon)>0$ so that the following holds for any $(X,p)\in\Alex^n(-1)$. Suppose that there is a $(k,\delta)$-splitting function $u\colon B_{50}(p)\to \dR^k$. If $x_i\in\cS^k_{\epsilon,\,\beta \Lambda r_i}$ and $\{B_{r_i}(x_i)\}$ are disjoint for all $i\in\mathds I$, then for any $z\in\dR^k$, we have
%
%In particular, we have
%\begin{align}
%  \Big|\Big\{x_i\in\cS^k_{\epsilon,\,\beta \Lambda r_i}\colon B_{\beta \Lambda r_i}(x_i)\cap u^{-1}(z)\neq \varnothing\Big\}\Big|<N(n,\epsilon, \Lambda).
%  \label{t:finite non-sym scale.se1}
%\end{align}
%If $r_i\equiv r$, then
%\begin{align}
%  \Big|\Big\{x_i\in\cS^k_{\epsilon,\,\Lambda r}\colon B_{\Lambda r}(x_i)\cap u^{-1}(z)\neq \varnothing\Big\}\Big|<N(n,\epsilon,\Lambda).
%  \label{t:finite non-sym scale.se2}
%\end{align}
%\end{corollary}

\begin{proof}%[Proof of Theorem \ref{t:finite non-sym scale}]
Let $\bar x_i\in B_{\beta \Lambda r_i}(x_i)\cap u^{-1}(z)$ and $y_i\in B_{\beta \Lambda r_i}(x_i)\cap\cS^k_{\epsilon,\,\beta \Lambda r_i}$. There exists $\eta(n,\epsilon)>0$ such that $\bar x_i\in \cS^k_{\eta,10\beta \Lambda r_i}$, since $B_{10\beta \Lambda r_i}(\bar x_i)\supseteq B_{\beta \Lambda r_i}(y_i)$ and $B_{\beta \Lambda r_i}(y_i)$ is not $(k,\epsilon)$-splitting. Moreover, we have that $B_{r_i/2}(\bar x_i)$ are disjoint, because $B_{r_i/2}(\bar x_i)\subseteq B_{r_i}(x_i)$. Estimate (\ref{t:finite non-sym scale.se1-0}) follows by applying Lemma \ref{l:finite non-sym scale-1} to the collection $\{B_{r_i/2}(\bar x_i)\}$.

To prove (\ref{t:finite non-sym scale.se2-0}), one can go through the proof of Lemma \ref{l:finite non-sym scale-1} and (\ref{t:finite non-sym scale.se1-0}) with small modifications, or use the following re-covering arguments. Let $r'=r/\beta $. Then we have $B_{\beta  \Lambda r'}(x_i)=B_{\Lambda r}(x_i)$. The given conditions $B_{\Lambda r}(x_i)\cap\cS^k_{\epsilon,\,\Lambda r}\neq\varnothing \text{ and } B_{\Lambda r}(x_i)\cap u^{-1}(z)\neq \varnothing$ are equivalent to $B_{\beta  \Lambda r'}(x_i)\cap\cS^k_{\epsilon,\,\beta  \Lambda r'}\neq\varnothing$ and $B_{\beta  \Lambda r'}(x_i)\cap u^{-1}(z)\neq\varnothing$, respectively. The collection $\{B_{r'}(x_i)\}$ is not disjoint, so we can't use (\ref{t:finite non-sym scale.se1-0}) directly.  However, note that if $B_{r'}(x_i)\cap B_{r'}(x_j)\neq\varnothing$, then $B_r(x_j)\subseteq B_{2r'}(x_i)$. Because $\{B_{r}(x_i)\}$ are disjoint, for every $i$, there are at most $N(n,r'/r)=N(n,\beta )$ balls $B_{r'}(x_j)$ such that $B_{r'}(x_i)\cap B_{r'}(x_j)\neq\varnothing$. Therefore, the collection $\{B_{r'}(x_i)\}$ can be written as the union of $N(n,\beta )$ disjoint collections. Then the result follows from (\ref{t:finite non-sym scale.se1-0}).
%we apply the same covering procedure in Step 2. For $j\in\mathcal D_k$, we stop the procedure if $\beta  \sigma_{y_j}r_{x_j}\le r$. Therefore, for such an index $j$, we have $\left|\mathds I\left(B_{\lambda_{x_j}r_{x_j}}(x_j)\cap B_{\sigma_{y_j}r_{x_j}}(y_j)\right)\right|<N(n,\beta )$ and the statement follows.
\end{proof}

Let us now remark on a standard covering argument.  Let $\mathfrak{B}$ be a collection of sets. The intersection number $\mathcal N(\mathcal B)$ of $\mathfrak{B}$ is the minimum number $k$ so that $B_1\cap B_2\cap\dots\cap B_{k+1}=\varnothing$ for any $B_1,B_2,\dots,B_{k+1}\in \mathfrak{B}$. In particular, if $\mathcal N(\mathcal B)=1$, then $\mathfrak{B}$ is a disjoint collection. We have the following easy lemma:

\begin{lemma}\label{l:intersection_packing_est}
  Let $B_R(0)\subset \dR^k$ and $\mathfrak{B}=\{B_{r_i}(x_i)\subseteq B_R(0)\}$ be a collection of balls. If the intersection number $\mathcal N(\mathfrak{B})\le N<\infty$, then $\sum r_i^k< N\cdot C(k)R^k$.
\end{lemma}

%\begin{proof}
%  Let us construct a Vitali covering $\{B_{R_j}(y_j)\}$ in the classical way. Let $B_{R_1}(y_1)$ be the ball with maximum radii. Suppose that $\{B_{R_j}(y_j)\}_{j=1}^{\ell}$, has been selected, where $\{y_j\}\subseteq\{x_i\}$. Let $B_{R_{\ell+1}}(y_{\ell+1})$ be the ball with maximum radii among those balls which have no intersection with $\cup_{j=1}^{\ell} B_{R_j}(y_j)$.
%
%
%  Let $B_R(y)$ be one of the balls in the above Vitali covering and $\mathds I=\{i:r_i\le R \text{ and } B_{r_i}(y_i)\cap B_R(y)\neq \varnothing\}$. It suffices to show that
%  \begin{align}
%    \sum_{i\in \mathds I}r_i^k<C(N, k)R^k.
%    \label{e:overlap_packing_est.e1}
%  \end{align}
%
%  Note that we have the following property. If $\displaystyle \bigcap_{j=1}^K B_{r_j}(y_j)\neq\varnothing$, and $\displaystyle B_{r_j}(y_j)\cap B_R(p)\neq\varnothing$ for every $1\le j\le K$, then
%  \begin{align}
%    B_R(p)\cap\left(\bigcap_{j=1}^K B_{3^{K-1}r_j}(y_j)\right)\neq\varnothing.
%  \end{align}
%  This is because
%  \begin{align}
%    B_{r_1}(y_1)\subset B_{3r_2}(X_1)\subset B_{3^2r_3}(y_3)\subset \dots\subset B_{3^{K-1}r_K}(y_K),
%  \end{align}
%  if $r_1\le r_2\le\dots\le r_K$. Thus
%  \begin{align}
%    &B_R(y)\cap \left(\bigcap_{j=1}^K B_{3^{K-1}r_j}(y_j)\right)
%    \supset B_R(y)\cap B_{r_1}(y_1)\neq\varnothing.
%  \end{align}
%
%  The above property that the overlap number of $\{B_{3^{-N}r_i}(y_i), \, i\in \mathds I\}$ is at most $N-1$. Since $B_{r_i}(y_i)\subseteq B_{3R}(x)$, we have that (\ref{e:overlap_packing_est.e1}) follows from an induction on $N$.
%\end{proof}

Now let us prove a local version of Theorem \ref{t:packing_estimate}.

\begin{lemma}[Local packing estimate] \label{l:packing_estimate local}
For any $n\in\dN$, $\epsilon>0$, $R\le 1$ and $\Lambda\ge 1$, there exists $\delta(n,\epsilon)>0$ and $\beta (n,\epsilon)>0$ so that the following hold for any $(X,p)\in\Alex^n(-1)$, provided that $B_{500R}(p)$ is $(k,\delta)$-splitting.
\begin{enumerate}
  \renewcommand{\labelenumi}{(\roman{enumi})}
  \item If $x_i\in \cS^k_{\epsilon,\,\beta   r_i}\cap B_R(p)$ with $r_i\le R$ and $\{B_{r_i}(x_i)\}$ are disjoint for all $i\in\mathds I$, then $\displaystyle\sum_{i\in\mathds I} r_i^k < C(n,\epsilon)R^k$.
\item If $x_i\in \cS^k_{\epsilon,\,\Lambda r}\cap B_R(p)$ with $r\le R$ and $\{B_{r}(x_i)\}$ are disjoint for all $i\in\mathds I$, then $|\mathds I| < C(n,\epsilon,\Lambda)(R/r)^{k}$.
\end{enumerate}
%
%
%  For any $\epsilon>0$ and $\Lambda\ge 1$, there exist $\delta(n,\epsilon)>0$ and $\beta (n,\epsilon, \Lambda)>0$ so that the following hold.
%    \begin{enumerate}
%  \renewcommand{\labelenumi}{(\roman{enumi})}
%  \item Let $\{B_{r_i}(x_i)\}$ be a disjoint collection with $r_i\le \Lambda R$ and $x_i\in \cS^k_{\epsilon,\beta \Lambda r_i}\cap B_R(p)$\, for every $i$. Then $\sum r_i^k<C(n,\epsilon,\Lambda)\,R^k$.
%  \item Let $\{B_{r}(x_i)\}_{i=1}^N$ be a disjoint collection with $x_i\in \cS^k_{\epsilon,\Lambda r}\cap B_R(p)$\, for every $i$. Then $N<C(n,\epsilon, \Lambda)\,(R/r)^k$.
%  \end{enumerate}
\end{lemma}

\begin{proof}
 We prove (i) only and the proof of (ii) is similar, modulo (\ref{t:finite non-sym scale.se2-0}). By Proposition \ref{p:split-theory}, there is a $\delta_1$-splitting map $u\colon B_{50R}(p)\to \dR^k$. Assume $u(p)=0^k\in \dR^k$.

  Consider the collection of balls $\mathfrak{B}=\{B_{\frac12\beta r_i}(u(x_i)), i\in\mathds I\}$ in $\dR^k$. Because $u$ is 1-Lipschitz, we have that $B_{\frac12\beta r_i}(u(x_i))\subseteq B_{2R}(0^k)$. Given $z\in \dR^k$, let $\mathds I_z=\{i\in\mathds I: z\in B_{\frac12\beta r_i}(u(x_i))\}$.
  By Proposition \ref{p:split-theory} again, we have $u^{-1}(z)\cap B_{\beta r_i}(x_i)\neq\varnothing$. It follows from (\ref{t:finite non-sym scale.se1-0}) that $|\,\mathds I_z\,|\le N(n,\epsilon)$. This shows that the intersection number $\mathcal N(\mathfrak{B})\le N(n,\epsilon)$. Then the desired result follows from Lemma \ref{l:intersection_packing_est}.
\end{proof}

%We now prove the following statement which implies Theorem \ref{t:packing_estimate}.
%
%\begin{theorem}[Packing estimate] \label{t:packing_estimate -r}
% For any $n,\epsilon>0$, there exists $C(n,\epsilon)>0$ and $\beta (n,\epsilon)>0$ such that the following holds. Let $(X,p,d)\in\Alex^n(-1)$ and $0\le k\le n$.
%     \begin{enumerate}
%  \renewcommand{\labelenumi}{(\roman{enumi})}
%  \item Let $\{B_{r}(x_i)\}_{i=1}^N$ be a disjoint collection with $r\le 1$ and $x_i\in \cS^k_{\epsilon,100r}\cap B_1(p)$ for every $i$. Then $N<C(n,\epsilon)\,r^{-k}$.
%  \item Let $\{B_{r_i}(x_i)\}$ be a disjoint collection with $r_i\le 1$ and $x_i\in \cS^k_{\epsilon,\beta r_i}\cap B_1(p)$\, for every $i$. Then $\sum r_i^k<C(n,\epsilon)$.
%  \end{enumerate}
%
%\end{theorem}

Now we prove Theorem \ref{t:packing_estimate} by showing the following stronger statement.

\begin{theorem}[Packing estimate]\label{t:packing_estimate2}
Lemma \ref{l:packing_estimate local} still holds if the splitting assumption is dropped.
%For any $n\in\dN$, $\epsilon>0$ and $\Lambda\ge 1$, there exists $\beta (n,\epsilon)>0$ so that the following hold for any $(X,p)\in\Alex^n(-1)$ and any disjoint collection $\{B_{r_i}(x_i)\}$.
%\begin{enumerate}
%  \renewcommand{\labelenumi}{(\roman{enumi})}
%  \item $\Big|\Big\{i: r_i=r,\, x_i\in \cS^k_{\epsilon,\,\Lambda r}(X)\cap B_1(p)\Big\}\Big|<C(n,\epsilon,\Lambda)r^{-k}$.
%  \item $\displaystyle\sum_{i\in\mathds I} r_i^k < C(n,\epsilon)$, where $\mathds I=\Big\{i: r_i\le 1, \,x_i\in \cS^k_{\epsilon,\,\beta  r_i}(X)\cap B_1(p)\Big\}$.
%\end{enumerate}
\end{theorem}

\begin{proof}
We prove by induction on $k$. Let $0<\delta'(n,\epsilon)<\delta(n,\epsilon)<\delta_1(n,\epsilon)<\epsilon$ be determined latter. The constant $C$ may vary line by line. Lemma \ref{l:packing_estimate local} proves the case for $k=0$ as well as the case that $B_{500R}(p)$ is $(k+1,\delta_1)$-splitting. Assume that (i) and (ii) are true for $k<n$. We will prove them for $k+1$, assuming that $B_{500R}(p)$ is not $(k+1,\delta_1)$-splitting.

Not losing generality, assume $R=\frac{1}{500}$. That is, $B_1(p)$ is not $(k+1,\delta_1)$-splitting. We begin with a decomposition of $B_1(p)$. Let $R_\alpha=2^{-\alpha}$, $\alpha\in \dZ$. Recall the definition of the weak $(k,\delta)$-singular set $\widetilde\cS^k_{\delta,r}$ in (\ref{d:WS}).  By Proposition \ref{p:sing.equv}, we have $B_1(p)\subseteq \widetilde\cS^k_{\delta,10}$. Thus
  \begin{align}
    B_1(p)\setminus \widetilde\cS^k_\delta
    \subseteq \widetilde\cS^k_{\delta,10}\setminus \widetilde\cS^k_\delta
    \subseteq\bigcup_{\alpha=-4}^\infty \left(\widetilde\cS^k_{\delta,\,R_{\alpha}}\setminus \widetilde\cS^k_{\delta,\,R_{\alpha+1}}\right).
    \label{t:packing_estimate.e2}
  \end{align}
%  Therefore, for every $y\in B_1(p)\setminus \widetilde\cS^k_\delta$, there is $\alpha_y\in\dZ$ so that $y\in \widetilde\cS^k_{\delta,\,R_{\alpha_y}}\setminus \widetilde\cS^k_{\delta,\,R_{\alpha_y+1}}$.
For each $\alpha$, let
\begin{align}
  \left\{B_{\rho_\alpha}(y^{\,\alpha}_j), \,j\in \mathds J_\alpha\right\}
  \subseteq
  \left\{B_{\frac1{20}R_{\alpha}}(y), \, y\in \widetilde\cS^k_{\delta,\,R_{\alpha}}\setminus \widetilde\cS^k_{\delta,\,R_{\alpha+1}}\right\}
\end{align}
be a Vitali covering of $\left(\widetilde\cS^k_{\delta,\,R_{\alpha}}\setminus \widetilde\cS^k_{\delta,\,R_{\alpha+1}}\right)$, for which $\{B_{\frac15\rho_\alpha}(y_j)\}$ are disjoint but $\{B_{\rho_\alpha}(y_j)\}$ is a covering. %For each $\alpha$, let $\mathds J_\alpha=\{j: \rho_j= \frac{1}{20}R_{\alpha}\}$.
A useful property for this decomposition is that for each $y\in\{y_j^{\,\alpha}\}$, we have that
$B_{20\rho_\alpha}(y)=B_{R_\alpha}(y)$ is not $(k+1,\delta)$-splitting, but $B_{10\rho_\alpha}(y)=B_{\frac12R_{\alpha}}(y)$ is $(k+1,\delta)$-splitting.

   We first prove (ii), which will be needed in the proof of (i).  %Let $\mathds I=\Big\{i: r_i=r,\, x_i\in \cS^k_{\epsilon,\,\Lambda r}(X)\cap B_1(p)\Big\}$ and only consider $i\in\mathds I$.
   By the inductive hypothesis, we only need to consider the collection of balls $\Big\{B_{r}(x_i): x_i\in \cS^{k+1}_{\epsilon,\,\Lambda r}$ but $x_i\notin \cS^k_{\delta',\, r}\Big\}$, where $\delta'=\delta'(n,\delta)>0$ will be determined latter. For each $j\in\mathds J_\alpha$, because $y_j^{\,\alpha}\in \widetilde\cS^k_{\delta,\,R_{\alpha}}\subseteq \cS^k_{\delta,\,R_{\alpha}}$, by the inductive hypothesis, we have an upper bound on the number of these balls:
  \begin{align}
     |\mathds J_\alpha|\le C(n,\epsilon) R_\alpha^{-k}.
     \label{l:packing_estimate r-local.e1}
  \end{align}

Recall that $x_i\in \cS^k_{\epsilon,\,\Lambda r}\cap B_{1/500}(p)$ with $r\le 1/500$ and $\{B_{r}(x_i)\}$ are disjoint. Given $j\in\mathds J_\alpha$, let $\mathds I_j^{\,\alpha}=\{i: x_i\in B_{\rho_\alpha}(y_j^{\,\alpha})\}$. We claim that if $\rho_\alpha <r/1000 $, then $\mathds I_j^{\,\alpha}=\varnothing$ for every $j$. Suppose $\rho_\alpha <r/1000 $ but there is $i\in\mathds I_j^{\,\alpha}$ for some $j$. Note then that $R_\alpha=20\rho_\alpha<\frac1{50}r$, we have $B_{r/5}(x_i)\supseteq B_{20\rho_\alpha}(y_j^{\,\alpha})$. Because $B_{20\rho_\alpha}(y_j^{\,\alpha})$ is not $(k+1,\delta)$-splitting, we have $x_i\in\cS^k_{\delta', \,r}$, for some $\delta'(n,\delta)>0$. This contradicts to the assumptions.

Now for each $i\in\mathds I_j^{\,\alpha}$, we have $x_i\in \cS^{k+1}_{\epsilon,\,\Lambda r}$, $r\le 1000\rho_\alpha=50 R_\alpha$, and $B_{10\rho_\alpha}(y_j^{\,\alpha})=B_{\frac12R_{\alpha}}(y)$ is $(k+1, \delta)$-splitting. By Lemma \ref{l:packing_estimate local} (ii) we have
    \begin{align}
     |\mathds I_j^{\,\alpha}|\le C(n,\epsilon, \Lambda)(R_\alpha/r)^{\,k+1}.
     \label{l:packing_estimate r-local.e2}
  \end{align}
Because $\cup_{\alpha\ge -4}\{B_{\rho_\alpha}(y_j^{\,\alpha})\}$ is a covering of $B_1(p)\setminus \widetilde\cS^k_\delta\supseteq B_1(p)\setminus \widetilde\cS^k_{\delta'}\supseteq\{x_i:i\in\mathds I\}$, by (\ref{l:packing_estimate r-local.e1}) and (\ref{l:packing_estimate r-local.e2}), we have
  \begin{align*}
    |\mathds I|&\le \sum_{\frac1{50}r\,\le R_\alpha\,\le 10}\;\sum_{j\in\mathds J_\alpha} |\mathds I_j^{\,\alpha}|
    \\
    &\le \sum_{\alpha=-4}^\infty
     C(n,\epsilon, \Lambda) R_\alpha^{-k} (R_\alpha/r)^{\,k+1}\le C(n,\epsilon, \Lambda)r^{-(k+1)}.
  \end{align*}

We prove (i) in a similar way. By the inductive hypothesis, we only need to consider the balls $\Big\{B_{r_i}(x_i): x_i\in \cS^{k+1}_{\epsilon,\,\beta  r_i} \text{ but } x_i\notin \cS^k_{\delta',\, \beta r_i}\Big\}$, for some $\delta'(n,\delta)>0$. Given $j\in\mathds J_\alpha$, let $\mathds I_j^{\,\alpha}=\{i: x_i\in B_{\rho_\alpha}(y_j^{\,\alpha})\}$. We claim that for every $i\in \mathds I_j^{\,\alpha}$, we have $r_i\le \frac{1000}{\beta }\rho_\alpha$. If this is not true, then $B_{\beta r_i/5}(x_i)\supseteq B_{20\rho_\alpha}(y_j^{\,\alpha})$. Because $B_{20\rho_\alpha}(y_j^{\,\alpha})$ is not $(k+1,\delta)$-splitting, we have that $B_{\beta r_i}(x_i)$ is not $(k+1,\delta')$-splitting for some $\delta'=\delta'(n,\delta)>0$. Thus $x_i\in \cS^k_{\delta', \,\beta r_i}$, which contradicts to the assumptions.

Note that $x_i\in \cS^{k+1}_{\epsilon,\,\beta  r_i}$ and $B_{10\rho_\alpha}(y_j^{\,\alpha})=B_{\frac12R_{\alpha}}(y)$ is $(k+1,\delta)$-splitting. We can apply Lemma \ref{l:packing_estimate local} (i) and get
  \begin{align}
     \sum_{i\in \mathds I_j^{\,\alpha}} r_i^{\,k+1}\le C(n,\epsilon)\rho_\alpha^{\,k+1}= C(n,\epsilon)R_\alpha^{\,k+1}.
     \label{l:packing_estimate r-local.e3}
  \end{align}
%By (i), we have that (\ref{l:packing_estimate r-local.e1}) holds.
Note that (\ref{l:packing_estimate r-local.e1}), which was proved in the course of proving (ii), still holds. Combine (\ref{l:packing_estimate r-local.e1}) and (\ref{l:packing_estimate r-local.e3}). We have
  \begin{align}
    \sum_{i\in\mathds I} r_i^{\,k+1}
    &\le\sum_{\alpha=-4}^\infty\sum_{j\in\mathds J_\alpha} \sum_{i\in\mathds I_j^{\,\alpha}} r_i^{\,k+1}
    \notag
    \\
    &\le \sum_{\alpha=-4}^\infty\sum_{j\in\mathds J_\alpha} C(n,\epsilon)R_\alpha^{\,k+1}
    \notag
    \\
    &\le \sum_{\alpha=-4}^\infty C(n,\epsilon)R_\alpha^{\,-k}R_\alpha^{\,k+1}
    \notag
    \\
    &\le \sum_{\alpha=-4}^\infty C(n,\epsilon) R_\alpha
    \le C(n,\epsilon).
  \end{align}
\end{proof}

\section{Sharpness of the Rectifiability}\label{s:examples}

In this section we prove Theorem \ref{t:sharp rect.cantor}. Let us begin with a smoothing lemma.

  \begin{lemma}\label{l:F convex f}
    Let $\mathbb U\subset\dR^n$ be a compact convex subset and $f\colon \mathbb U\to\dR$ be a strictly convex function. Let $\Omega=\cup_{i=1}^\infty \Omega_i$, where $\Omega_i$ are disjoint open convex subsets in $\mathbb U$. For any $\delta>0$, there exists a strictly convex function $F\colon\mathbb U\to\dR$ such that the following hold.
    \begin{enumerate}
    \renewcommand{\labelenumi}{(\roman{enumi})}
    \setlength{\itemsep}{1pt}
    \item $F|_\Omega$ is $C^\infty$.
    \item $F|_{\mathbb U\setminus\Omega}=f|_{\mathbb U\setminus\Omega}$ and $|F-f|<\delta$ on $\Omega$,
    \item For any $x\notin\Omega$ and any vector $v$, it holds that
    \begin{align}
    \underset{t\to 0^+}\lim\frac{F(x+tv)-F(x)}{t}
    =\underset{t\to 0^+}\lim\frac{f(x+tv)-f(x)}{t}.
    \label{l:F convex f.e1}
    \end{align}
    In particular, if $Df(x)$ exists at $x\notin\Omega$, then $DF(x)=Df(x)$.

%    \item For any $x\in\Omega$, if there is an open set $U\ni x$ such that $f|_U$ is $C^1$, then $|D F(x)- Df(x)|<\delta$. \blue{not correct}
\end{enumerate}
\end{lemma}

  \begin{proof}
  Let
  \begin{align}
    \epsilon(x)= e^{-\frac{\delta}{d(x,\,\mathbb U\setminus \Omega)}}
  \end{align}
  be an error function defined on $\Omega$.
  By Theorem 1.1 in \cite{GW79}, for each $i$, there exists a strictly $C^\infty$ convex function $g_i\colon \Omega_i\to\dR$ such that for any $x\in\Omega_i$, we have
  \begin{align}
    |f(x)-g_i(x)|\le\epsilon(x).
    \label{t:sharp rect.e03}
  \end{align}
  Let $F\colon \mathbb U\to\dR$ be the gluing of all of $g_i$ and $f|_{\mathbb U\setminus \Omega}$. That is,
  \begin{align}
    F(x)=
    \renewcommand{\arraystretch}{1.5}
   \left\{\begin{array}{@{}l@{\quad}l@{}}
    g_i(x), & \text{ if \quad} x\in \Omega_i;
      \\
      f(x), & \text{ if \quad} x\notin \Omega.
  \end{array}\right.
    \label{t:sharp rect.e10}
  \end{align}
  It is obvious that (i) and (ii) are satisfied. The following estimates (\ref{t:sharp rect.e8-1}) and (\ref{t:sharp rect.e8}) imply (iii). If $x, y\notin\Omega$, it is obvious that
    \begin{align}
    \Big||F(x)-F(y)|- |f(x)-f(y)|\Big|=0.
    \label{t:sharp rect.e8-1}
    \end{align}
  For any $x\notin\Omega$ and $y\in \Omega$, we have $y\in\Omega_i$ for some $i$ and thus
  \begin{align}
    \Big||F(x)-F(y)|- |f(x)-f(y)|\Big|
    &=\Big||f(x)-g_i(y)|- |f(x)-f(y)|\Big|
    \notag
    \\
    &\le|g_i(y)-f(y)|
    \notag
    \\
    &\le e^{-\frac{\delta}{d(y,\,\mathbb U\setminus \Omega)}}
    \notag
    \\
    &\le e^{-\frac{\delta}{d(x,\,y)}}.
    \label{t:sharp rect.e8}
  \end{align}

  It's clear that $F$ is strictly convex on each of $\Omega_i$. It remains to show that $F$ is strictly convex, moving out from $\Omega_i$. We need the following two lemmas which we will outline the proof latter. They are well known to the experts.

  \begin{lemma}\label{l:convex-der}
    A Lipschitz function $h\colon[a,b]\to \dR$ is convex if and only if for any non-negative smooth function $\phi\colon[a,b]\to\dR$, it holds that
  \begin{align}
    \int_a^b h'(t)\,\phi'(t)\,dt\le h'_-(b)\,\phi(b)-h'_+(a)\,\phi(a).
    \label{t:sharp rect.e4}
  \end{align}
  Here $h'_\pm$ denote the one-sided derivatives.
  \end{lemma}

  \begin{lemma}\label{l:convex-combine}
    Let $h\colon[a,c]\to\dR$ be a Lipschitz function and $b\in[a,c]$. If $h|_{[a,b]}$ and  $h|_{[b,c]}$ are both convex functions and $h'_-(b)\le h'_+(b)$, then $h$ is a convex function over $[a,c]$.
  \end{lemma}

  Now we show that $F$ is a convex function. It is obvious that $F$ is locally convex for any $x\notin \partial\Omega_i$. For $x\in\partial \Omega_i$, we show that $F$ is convex along each line passing through $x$ in $\mathbb U$. Let $\gamma(s)=x+sv$, $s\in(-\epsilon, \epsilon)$ be a unit speed geodesic in $\mathbb U$ and $h(s)=F(x+sv)$. By (\ref{l:F convex f.e1}) and the fact that $f$ is convex, we have
  \begin{align}
    h'_-(0)&=\lim_{t\to 0^-}\frac{F(x+tv)-F(x)}{t}
    \notag\\
    &=\lim_{t\to 0^-}\frac{f(x+tv)-f(x)}{t}\le \lim_{t\to 0^+}\frac{f(x+tv)-f(x)}{t}
    \notag\\
    &=\lim_{t\to 0^+}\frac{F(x+tv)-F(x)}{t}
    =h'_+(0).
  \end{align}
  Then the convexity of $F$ follows from Lemma \ref{l:convex-combine}. By (iii), $F$ is also strictly convex.
  \end{proof}

  \begin{proof}[Proof of Lemma \ref{l:convex-der}]
    The necessity is obvious. To prove the sufficiency it is sufficient to verify
  $$\frac{h(t_2)-h(t_1)}{t_2-t_1}-\frac{h(t_3)-h(t_2)}{t_3-t_2}\le 0
  $$
  for every $a\le t_1< t_2< t_3\le b$. This can be proved by a direct computation with $\phi(t)$ chosen as a smooth approximation to
  $$\psi(t)
     =\renewcommand{\arraystretch}{1.5}
     \left\{\begin{array}{@{}l@{\quad}l@{}}
    0, & \text{ if \quad} t\le t_1, \\
    \frac{t-t_1}{t_2-t_1}, & \text{ if \quad} t_1<t\le t_2, \\
    \frac{t_3-t}{t_3-t_2}, & \text{ if \quad} t_2<t< t_3, \\
    0, & \text{ if \quad} t\ge t_3.
  \end{array}\right.
  $$
  \end{proof}

  \begin{proof}[Proof of Lemma \ref{l:convex-combine}]
    By Lemma \ref{l:convex-der}, for any non-negative smooth function $\phi:[a,b]\to[0,\infty)$, we have
  \begin{align}
    &\int_{a}^{\,b} h'(t)\,\phi'(t)\,dt\le h'_-(b)\,\phi(b)-h'_+(a)\,\phi(a),
    \label{t:sharp rect.e5}
    \\
    &\int_b^{\,c} h'(t)\,\phi'(t)\,dt\le h'_-(c)\,\phi(c)-h'_+(b)\,\phi(b).
    \label{t:sharp rect.e6}
  \end{align}
  Sum up the two inequalities and apply Lemma \ref{l:convex-der} again, we get the desired result.
  \end{proof}

\begin{proof}[Proof of Theorem \ref{t:sharp rect.cantor}]

%It suffices to construct an example $Y$ for $k=1$ and $n=3$. For $n>3$, the product space $Y\times\dR^{n-3}$ is the desired Alexandrov space.

Let $Z=\bar B_1(O)\subset \dR^2$ be a closed unit disk centered at $p$. Fix $0<\delta<1$ and define a strictly concave function on $Z$:
\begin{align}
   f_0(z) &= \renewcommand{\arraystretch}{1.5}
   \left\{\begin{array}{@{}l@{\quad}l@{}}
    \sqrt{d(z,\partial Z)} & \text{ if \quad} d(z,\partial Z)\le \frac14,
    \\
    \delta\cdot \sqrt{d(z,\partial Z)}+(1-\delta)\cdot \frac12 & \text{ if \quad} d(z,\partial Z)> \frac14.
  \end{array}\right.
  \label{t:sharp rect.e1}
\end{align}
Let $Z_t=\{z\in Z: f_0(z)\ge t\}$ be the sub-level set. We denote the subgraph of $f\colon Z\to\dR^+$ by
$$G_{Z,f}=\Big\{(z,t)\in Z\times\dR:0\le t\le f(z) \Big\}.$$

Because $f_0$ is strictly concave, we have $X_0=G_{Z,f_0}\in\Alex^3(0)$ with boundary.  See Figure \ref{figure 3} below.
  \begin{figure}[!h]
  \includegraphics[scale=1]{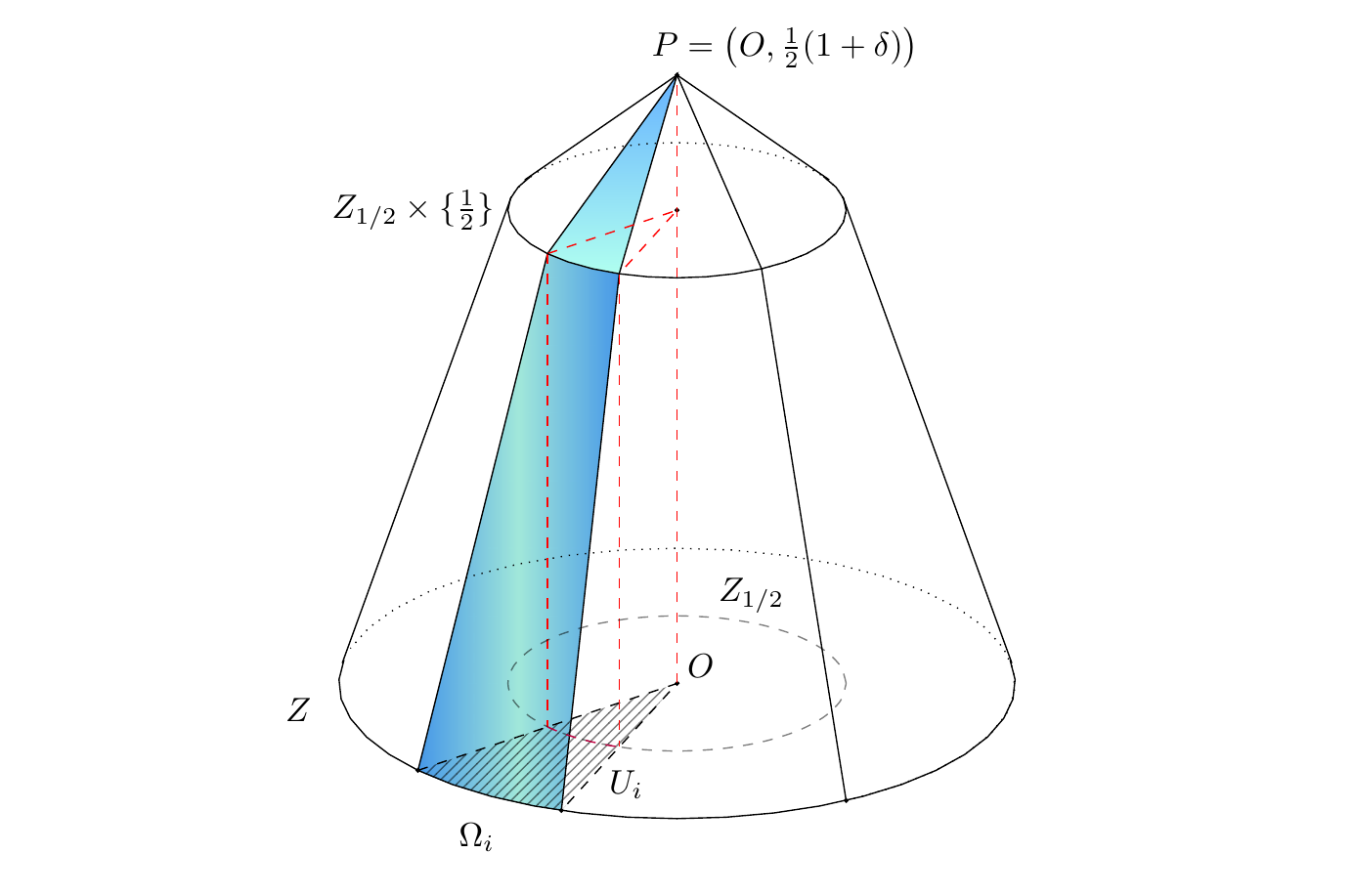}
  \caption{$X_0=G_{Z,f_0}\in\Alex^3(0)$}
  \label{figure 3}
  \end{figure}
For $\delta>0$ small the following hold:
\begin{enumerate}
  \renewcommand{\labelenumi}{(\arabic{enumi})}
  \setlength{\itemsep}{1pt}
  \item $\cS(X_0)=\partial X_0$,
  \item $ \cS^1_\epsilon(X_0)\setminus \cS^{0}(X_0)=\left(\partial Z_{1/2}\times\{\frac12\}\right)\cup \left(\partial Z\times\{0\}\right)$
  \item $\cS^{0}(X_0)=\{P\}=\left\{\left(O,\frac12(1+\delta)\right)\right\}$ is the tip of the graph.
\end{enumerate}

Not losing generality, let $T\subseteq \partial Z_{1/2}$ be any closed subset. Then $\displaystyle\partial Z_{1/2}\setminus T=\cup_{i=1}^\infty U_i$ is a union of disjoint open intervals.  Let $\Omega_i$ be the open sectors in $Z$ corresponding to the arc $U_i$. That is, $\Omega_i=\{x\in Z^\circ\colon \text{ray } \lambda\cdot\overrightarrow{Ox}\cap U_i\neq\varnothing\}$, as the shaded region in Figure \ref{figure 3}. Clearly, $\{\Omega_i\}$ is a collection of disjoint open convex sets.

Now apply Lemma \ref{l:F convex f} to $f_0:Z\to \dR$ on $\cup_{i=1}^\infty \Omega_i$ to obtain a strictly convex function $f_1:Z\to \dR$ which is smooth on $\cup_{i=1}^\infty \Omega_i$ and $f_1=f_0$ away from $\cup_{i=1}^\infty \Omega_i$. Now consider the new subgraph $X_1=G_{Z,\,f_1}\in\Alex^3(0)$. Note that if $f_1$ is smooth at a point $x\in Z^\circ$, then the tangent cone of $X_1$ at $(x,f_1(x))\in\partial X_1$ is a three dimensional half space. If $x\in \partial Z_{1/2}\setminus \cup_{i=1}^\infty \Omega_i\equiv T$, then the tangent cone at $(x,f_1(x))\in\partial X_1$, which is isometric to the tangent cone of $X_0$ at $(x,f_0(x))$, splits off only $\dR^1$. Therefore we have
\begin{enumerate}
  \renewcommand{\labelenumi}{(\arabic{enumi})}
  \setlength{\itemsep}{1pt}
    \setcounter{enumi}{3}
  \item $\cS(X_1)=\partial X_1$,
  \item $\cS^1_\epsilon(X_1)\setminus \cS^0(X_1)=\left(T\times\{\frac12\}\right)\cup \left(\partial Z\times\{0\}\right)$,
  \item $\cS^{0}(X_1)=\left\{P\right\}$ is the tip of the graph.
\end{enumerate}
A similar, but less involved, smoothing procedure can be performed in a small neighborhood of $\partial Z\times\{0\}$ and $P$ so that the resulting space $X_2\in\Alex^3(0)$ satisfies
\begin{enumerate}
  \renewcommand{\labelenumi}{(\arabic{enumi})}
  \setlength{\itemsep}{1pt}
    \setcounter{enumi}{6}
  \item $\cS(X_2)=\partial X_2$,
  \item $\cS^1_\epsilon(X_2)=T\times\{\frac12\}$,
  \item $\cS^{0}(X_2)=\varnothing$.
\end{enumerate}
Finally, we double $X_2$ and arrive at a boundary free space $Y\in\Alex^3(0)$ which satisfies
\begin{enumerate}
  \renewcommand{\labelenumi}{(\arabic{enumi})}
  \setlength{\itemsep}{1pt}
    \setcounter{enumi}{9}
  \item $\cS(Y)=\cS^1_\epsilon(Y)=T\times\{\frac12\}$;
  \item $\cS^{0}(Y)=\varnothing$.
\end{enumerate}
By performing similar smoothing procedures it is easy to verify that $Y$ can be realized as a non-collapsed limit of $3$-dimensional manifolds with $\sec\ge 0$.
\end{proof}

%\begin{example}\label{e:quant_strat_vol}
%  For any fixed $C(n,\epsilon)>0$, we construct an Alexandrov space $X\in\Alex^3(0)$ such that (\ref{t:quant_strat_vol.e1}) fails. Let $Z=\triangle pqs\subset \dR^2$ be a flat triangle plane, with $d(p,q)=1$, $d(p,s)\le2$, $\measuredangle pqs=\alpha$ and $\measuredangle qps=\beta<\alpha$. Let $X=Z\times[0,5]$.
%
%  \begin{center}
%  \includegraphics[scale=1.2]{Example3.pdf}
%  \end{center}
%
%  %\begin{figure}
%%\includegraphics[scale=1.2]{Example3.pdf}
%%\caption{Example 3}
%%\end{figure}
%
%  Consider $\cS^1_\epsilon(X)$ in $B_{10}(p)=X$. For $0<\epsilon\ll\pi-\alpha$ and $0<r\ll\beta$, we have
%  $\mathcal H^3(B_r(\cS^1_\epsilon))\ge \alpha r^2$ and $\mathcal H^3(X)\le 10\,\beta$. If (\ref{t:quant_strat_vol.e1}) is true, then we have
%  \begin{align}
%    \alpha r^2\le C(\epsilon)\, \beta r^2,
%  \end{align}
%  which is a contradiction when $\beta/\alpha\to 0$.
%
%
%
%
%
%
%\end{example}

\bibliographystyle{amsalpha}

\end{document}